\documentclass[11pt]{amsart}
\usepackage{amssymb,amsmath}

\theoremstyle{plain}
\newtheorem{thrm}{Theorem}[section]
\newtheorem{lemma}[thrm]{Lemma}

\newtheorem{cor}[thrm]{Corollary}
\newtheorem{rmrk}[thrm]{Remark}
\newtheorem{dfn}[thrm]{Definition}

\numberwithin{equation}{section}

\begin{document}
% begin top matter
% ********************* macroes needed for this paper ************************
\newcommand{\SL}{\mathcal L^{1,p}( D)}
\newcommand{\Lp}{L^p( Dega)}
\newcommand{\CO}{C^\infty_0( \Omega)}
\newcommand{\Rn}{\mathbb R^n}
\newcommand{\Rm}{\mathbb R^m}
\newcommand{\R}{\mathbb R}
\newcommand{\Om}{\Omega}
\newcommand{\Hn}{\mathbb H^n}
\newcommand{\N}{\mathbb N}
\newcommand{\aB}{\alpha B}
\newcommand{\eps}{\epsilon}
\newcommand{\BVX}{BV_X(\Omega)}
\newcommand{\p}{\partial}
\newcommand{\IO}{\int_\Omega}
\newcommand{\bG}{\mathbb{G}}
\newcommand{\bg}{\mathfrak g}
\newcommand{\Bux}{\mbox{Box}}
\newcommand{\al}{\alpha}
\newcommand{\til}{\tilde}
\newcommand{\nuX}{\boldsymbol{\nu}^X}
\newcommand{\bN}{\boldsymbol{N}}
\newcommand{\nh}{\nabla^H}
\newcommand{\Gp}{G_{D,p}}
\newcommand{\n}{\boldsymbol \nu}

% ****************************************************************************

\title[Non-divergence form parabolic equations, etc.]{Non-divergence form parabolic equations associated with non-commuting vector fields: Boundary behavior of nonnegative solutions}

\author{M. Frentz}
\address{Department of Mathematics
and mathematical Statistics\\ Ume{\aa } University\\ S-90187 Ume{\aa },
Sweden}
\email[Marie Frentz]{marie.frentz@math.umu.se}

\author{N. Garofalo}
\address{Department of Mathematics\\
Purdue University \\
West Lafayette IN 47907-1968}
\email[Nicola Garofalo]{garofalo@math.purdue.edu}
\thanks{Second author supported in part by NSF Grant DMS-07010001}
\author{E. G{\"o}tmark}
\address{Department of Mathematics
and mathematical Statistics\\ Ume{\aa } University\\ S-90187 Ume{\aa },
Sweden}
\email[Elin G{\"o}tmark]{elin.gotmark@math.umu.se}

\author{I. Munive}
\address{Department of Mathematics\\
Purdue University \\
West Lafayette IN 47907-1968}
\email[Isidro Munive]{imunive@math.purdue.edu}
\thanks{Fourth author supported in part by the second author's NSF Grant DMS-07010001}

\author{K. Nystr\"{o}m}
\address{Department of Mathematics
and mathematical Statistics\\ Ume{\aa } University\\ S-90187 Ume{\aa },
Sweden}
\email[Kaj Nystr\"{o}m]{kaj.nystrom@math.umu.se}

\date{\today}

%
% AMS information
%

\begin{abstract}
In a cylinder $\Omega_T=\Omega\times (0,T)\subset \R^{n+1}_+$ we study the boundary behavior of nonnegative
solutions of second order parabolic
equations of the form 
\[
Hu =\sum_{i,j=1}^ma_{ij}(x,t) X_iX_ju - \p_tu = 0, \ (x,t)\in\R^{n+1}_+,
\]
where $X=\{X_1,...,X_m\}$ is a system of $C^\infty$ vector fields in 
$\Rn$ satisfying H\"ormander's finite
rank condition \eqref{frc}, and $\Omega$ is a non-tangentially accessible domain with respect to the Carnot-Carath\'eodory distance $d$ induced by $X$. Concerning the matrix-valued function $A=\{a_{ij}\}$, we assume that 
it be real, symmetric and uniformly positive definite. Furthermore, we suppose that its entries $a_{ij}$ be H\"older continuous with respect to the parabolic distance associated with $d$. Our main results are: 1) a backward Harnack
inequality for nonnegative solutions vanishing on the lateral boundary
(Theorem \ref{T:back}); 2) the H\"older continuity  up to the boundary of the quotient of two nonnegative solutions
which vanish continuously on a portion of the lateral boundary (Theorem \ref{T:quotients}); 3) the doubling property for the parabolic measure associated with the operator $H$ (Theorem \ref{T:doubling}). These results generalize to the subelliptic setting of the present paper, those in Lipschitz cylinders by Fabes, Safonov and Yuan in [FSY] and [SY]. With one proviso: in those papers the authors assume that the coefficients $a_{ij}$ be only bounded and measurable, whereas we assume H\"older continuity with respect to the intrinsic parabolic distance.
\end{abstract}

\maketitle
% end top matter

%\tableofcontents

\newpage

\section{Introduction}\label{S:i}

Let $\Omega\subset\Rn$ be a bounded domain and consider the cylinder $\Omega_T=\Omega\times(0,T)\subset \R^{n+1}_+$, where $T>0$ is fixed. In this paper we establish a number of results concerning the boundary
behavior of non-negative solutions in $\Omega_T$ of second order parabolic equations of the type 
\begin{equation}\label{Hu}
Hu=Lu - \partial_t u = \sum_{i,j=1}^ma_{ij}(x,t)X_i X_j u - \partial_t u = 0. 
\end{equation}
Here, $X = \{X_1,...,X_m\}$ is a
system of $C^\infty$ vector fields in $\Rn$ satisfying H\"ormander's finite rank
condition, see \cite{H}:
\begin{equation}\label{frc}
\text{rank Lie}\ [X_1,...,X_m] \equiv n.
\end{equation} 
Concerning the $m\times m$ matrix-valued function $A(x,t)=\{a_{ij}(x,t)\}$ we assume that it be
symmetric, with bounded and measurable entries, and that there exists $\lambda \in [1,\infty)$ such that for every $(x,t)\in\R^{n+1}$, and $\xi\in\R^m$,
\begin{equation}\label{ell}
\lambda^{-1}|\xi|^2\leq \sum_{i,j=1}^ma_{ij}(x,t)\xi_i\xi_j\leq
\lambda|\xi|^2.
\end{equation}

When $m=n$ and $\{X_1,...,X_m\}=\{\partial_{x_1},...,\partial_{x_n}\}$, the operator $H$ in \eqref{Hu}
coincides with that 
studied in \cite{FSY} and \cite{SY}. However, in contrast with these papers, in which the coefficients were assumed only bounded and measurable, we will also assume that the entries of the matrix $A(x,t)$ be H\"older continuous with respect to the intrinsic parabolic distance associated with the system $X$. More precisely, we indicate with $d(x,y)$ the Carnot-Carath\'eodory distance, between $x,y\in\Rn$,
induced by $\{X_1,...,X_m\}$. We also let 
\[
d_p(x,t,y,s)=(d(x,y)^2+|t-s|)^{1/2}
\]
denote the parabolic distance associated with the metric $d$. Then, we assume that there exist $C>0$, and $\sigma\in (0,1)$, such that for $(x,t),\
(y,s)\in \R^{n+1}$,
\begin{equation} \label{hol}
|a_{ij}(x,t)-a_{ij}(y,s)|\leq C d_p(x,t,y,s)^\sigma,\ \ \ \ i,j\in\{1,..,m\}.
\end{equation}
The reason for imposing \eqref{hol} will be discussed below. 

Concerning the domain $\Omega$ we will assume that it be a NTA domain (non-tangentially accessible domain), with parameters $M$, $r_0$, in the sense of \cite{CG}, \cite{CGN4}, see Definition \ref{D:NTA} below. Under this assumption we can prove that all points on the parabolic boundary 
\[
\partial_p\Omega_T = S_T \cup (\Omega \times \{0\}),\ \ \ S_T = \partial \Omega \times (0,T),
\]
 of the cylinder $\Omega_T$ are regular for the Dirichlet problem for the operator $H$ in \eqref{Hu}. In particular, for any $f\in C(\partial_p\Omega_T)$, there exists  a unique Perron-Wiener-Brelot-Bauer solution $u=u_f^{\Omega_T}\in C(\overline \Omega_T)$ to the
Dirichlet problem 
\begin{eqnarray}  \label{dp}
Hu=0\mbox{ in $\Omega_T$,\ \ \ \  $u=f$ on $\partial_p\Omega_T$}.
\end{eqnarray}
Moreover, one can conclude that for every $(x, t)\in \Omega_T$ there exists a unique probability measure  $d\omega^{(x,t)}$ on $\partial_p \Omega_T$ for which
\begin{eqnarray}  \label{1.1xx}
u(x,t)=\int\limits_{\partial_p \Omega_T}f(y,s)d\omega^{(x,t)}(y,s).
\end{eqnarray}
Henceforth, we refer to $\omega^{(x,t)}$ as the $H$-\emph{parabolic measure} 
relative to $(x,t)$ and $\Omega_T$.

The metric ball centered at $x\in \Rn$ with radius $r>0$ will be indicated with 
\[
B_d(x,r)=\{y\in \Rn:\ d(x,y)<r\}.
\]
For  $(x,t)\in\R^{n+1}$ and $r>0$ we let
\[
C_r^-(x,t)=B_d(x,r)\times(t-r^2,t),\ \ \ C_r(x,t)=B_d(x,r)\times(t-r^2,t+r^2), 
\]
and we define 
\begin{equation}\label{surball}
\Delta(x,t,r)=S_T \cap C_r(x,t).
\end{equation} 
By Definition \ref{D:NTA} below, if $\Omega$ is a given NTA domain with parameters $M$ and $r_0$, for any $x_0\in\partial\Omega$, $0<r<r_0$,
there exists a non-tangential corkscrew, i.e., a point $A_r(x_0)\in\Omega$, such that 
\[
M^{-1}r<d(x_0,A_r(x_0))<r,\ \ \text{and}\ \ 
d(A_r(x_0),\partial\Omega)\geq M^{-1}r.
\]
In the following we let $A_r(x_0,t_0)=(A_r(x_0),t_0)$ whenever $(x_0,t_0)\in S_T$ and $0<r<r_0$. When
we say that a constant $c$ depends on the operator $H$ we mean that $c$ depends on the dimension $n$, the number of
vector fields $m$, the vector fields $\{X_1,...,X_m\}$, the constant $
\lambda $ in \eqref{ell} and the parameters $C, \sigma$ in \eqref{hol}. We let $\mbox{diam}
(\Omega) = \sup\{d(x,y)\mid x,y\in\Omega\}$ denote the diameter of $\Omega$. The following
theorems represents the main results of this paper.

\begin{thrm}[Backward Harnack inequality]\label{T:back}
Let $u$ be a nonnegative solution of $Hu=0$ in $\Omega_T$ vanishing continuously on $S_T$.  Let $0<\delta\ll 
\sqrt{T}$ be a fixed constant, let $(x_0,t_0)\in S_T$, $\delta^2\leq t_0\leq
T-\delta^2$, and assume that $r<\min\{r_0/2,\sqrt{(T-t_0-\delta^2)/4},\sqrt{(t_0-\delta^2)/4}\}$. Then, there exists
a constant $c=c(H,M,r_0,\mbox{diam}(\Omega),T,\delta)$, $1\leq c<\infty$,
such that for every $(x,t)\in \Omega_T\cap C_{r/4}(x_0,t_0)$ one has
\begin{equation*}
u(x,t)\leq cu(A_r(x_0,t_0)).
\end{equation*}
\end{thrm}

\begin{thrm}[Boundary H{\"o}lder continuity of
quotients of solutions]\label{T:quotients}
Let $u, v$ be nonnegative
solutions of $Hu=0$ in $\Omega_T$. Given $(x_0,t_0)\in S_T$, assume that $r<\min\{r_0/2,\sqrt{(T-t_0)/4},\sqrt{t_0/4}\}$.  
If $u, v$ vanish continuously on $\Delta(x_0,t_0,2r)$, then the quotient $v/u$ is H{\"o}lder continuous on the
closure of $\Omega_T\cap C_{r}^-(x_0,t_0)$.
\end{thrm}

\begin{thrm}[Doubling property of the $H$-parabolic
measure]\label{T:doubling} 
Let $K\geq 100$ and $\nu\in (0,1)$ be fixed constants. Let $(x_0,t_0)\in S_T$, and suppose that
$r<\min\{\nu r_0/2,\sqrt{(T-t_0)/4},\sqrt{t_0/4}\}.$
Then, there exists a constant $c=c(H,M,\nu,K,r_0)$, $1\leq c<\infty$, such that for every $(x,t)\in\Omega_T$, with $d(x_0,x)\leq K|t-t_0|^{1/2}$, $t-t_0\geq 16r^2$, one has 
\begin{eqnarray*}
\omega^{(x,t)}(\Delta(x_0,t_0,2r))\leq c\omega^{(x,t)}(\Delta(x_0,t_0,r)).
\end{eqnarray*} 
\end{thrm}

Concerning Theorems \ref{T:back}, \ref{T:quotients} and \ref{T:doubling}, we note that the study of the type of problems considered in this
paper has a long and rich history which, for uniformly parabolic equations in $\R^{n+1}$ (i.e., when in \eqref{Hu} one has $m=n$ and $\{X_1,...,X_m\}=\{\partial_{x_1},...,
\partial_{x_n}\}$), culminated with the celebrated papers of Fabes, Safonov and
Yuan \cite{FS}, \cite{FSY} and \cite{SY}. In these works the authors proved Theorem \ref{T:back}-\ref{T:doubling} for uniformly parabolic equations, both in divergence and non-divergence form, whose coefficients are only bounded and measurable.
We remark that, while these authors work in Lipschitz cylinders, one can easily
see that their proofs can be generalized to the setting  of bounded
NTA domains in the sense of \cite{JK}. While the works \cite{FSY}, \cite{SY} completed this line of
research for parabolic operators in non-divergence form, prior contributions by other researchers are contained in
\cite{FK}, \cite{FSt}, \cite{G}, \cite{KS}. For the
corresponding developments for second order parabolic operators in
divergence form we refer to \cite{FGS}, \cite{FS}, \cite{N}. For the elliptic theory,
for both operators in divergence and non-divergence form, we refer to \cite{B},
\cite{CFMS}, \cite{FGMS}, \cite{JK}. Finally, and for completion, we also note that second
order elliptic and parabolic operators in divergence form with singular
lower order terms were studied in \cite{KP} and \cite{HL}.

In the subelliptic setting of the present paper, i.e., when $m<n$ and $X =\{X_1,...,X_m\} $ is assumed to satisfy \eqref{frc}, much less is known. Several delicate new issues arise in connection with the intricate (sub-Riemannian) geometry associated with the vector fields, and the interplay of such geometry with the so-called characteristic points on the boundary of the relevant domain. In addition, the derivatives along the vector fields do not commute, and the commutators are effectively derivatives of higher order. For all these aspects we refer the reader to the works \cite{NS}, \cite{Citti}, \cite{D}, \cite{CG}, \cite{LU}, \cite{CGN3}, \cite{MM1}, \cite{MM2}, \cite{CGN4}, but this only represents a partial list of references.

In the stationary case, and for operators in divergence form, results similar to those in the present paper  have been obtained in
\cite{CG}, \cite{CGN3}, \cite{CGN4}, see also \cite{CGN1}, \cite{CGN2}, whereas for parabolic operators in divergence form the reader is referred to the recent paper by one of us \cite{Mu}. The methods in \cite{Mu}, however, extensively exploit the divergence structure of the operator and do not apply to the setting of the present paper.

We stress that for non-divergence form operators such as those treated in this paper, results such as 
Theorems \ref{T:back}-\ref{T:doubling} are new even for the case of stationary equations such as
\[
Lu = \sum_{i,j=1}^ma_{ij}(x)X_i X_j u = 0. 
\]
In view of these considerations our paper provides a novel contribution to the understanding of the
boundary behavior of solutions to parabolic equations arising from a system of non-commuting vector fields.

Concerning the proofs of Theorems \ref{T:back}-\ref{T:doubling} our approach is modeled on the ideas developed by
Fabes, Safonov and Yuan in \cite{FSY} and \cite{SY}. In fact, the ideas in those papers
have provided an important guiding line for our work. Yet, the arguments in \cite{FSY} and
\cite{SY} use mainly elementary principles like comparison principles, interior
regularity theory, the (interior) Harnack inequality, H{\"o}lder continuity
type estimates and decay estimates at the lateral boundary, for solutions
which vanish on a portion of the lateral boundary, as well as estimates for
the Cauchy problem and the fundamental solution associated to the operator
at hand. In this connection it is important that the reader keep in mind that when the matrix $A(x,t)=\{a_{ij}(x,t)\}$ in \eqref{Hu} has entries which are just bounded and measurable, then most of these results presently represent in our setting \emph{terra incognita}. More specifically, the counterparts of the Harnack inequality of Krylov and Safonov \cite{KS} and the Alexandrov-Bakel'man-Pucci type maximum principle due to Krylov \cite{Kr} presently constitute fundamental open questions.  

With this being said, our work uses heavily
the recent important results of Bramanti, Brandolini, Lanconelli and Uguzzoni
\cite{BBLU2}, see also \cite{BBLU1}, concerning the (interior) Harnack inequality, the
Cauchy problem and the existence and Gaussian estimates for fundamental solutions for
the non-divergence form operators $H$ defined in \eqref{Hu}. In fact, we assume \eqref{hol} precisely in order to be able to use results from \cite{BBLU2}.
We want to stress, however, that we have strived throughout the whole paper to provide proofs which are ``purely metrical''. By this we mean that, should the above mentioned counterpart of the results in \cite{Kr} and \cite{KS} become available, then our proofs would carry to the more general setting of bounded and measurable coefficients in \eqref{Hu} with minor changes.  

In closing we mention that the rest of the paper is organized as follows. Section 2 is of a preliminary
nature. In it we collect some notation and results concerning basic
underlying principles, and we also introduce the notion of
NTA domains following \cite{CG}. In section 3 we prove a number of basic
estimates concerning the boundary behavior of nonnegative solutions of \eqref{Hu}. In
addition we prove a number of technical lemmas which allow us to present the
proofs of Theorems \ref{T:back}-\ref{T:doubling} in a quite condensed manner. Finally, the
proofs of Theorems \ref{T:back} and \ref{T:quotients} will be presented in Section \ref{S:back}, whereas that of Theorem \ref{T:doubling} will be given in Section \ref{S:doubling}.

%%%%%%%%%%%%%%%%%%%%%%%%%%%

\section{Preliminaries}\label{S:prelim}

In this section we introduce some notation and state a number of preliminary
results for the operator $H$ defined in \eqref{Hu}. Specifically, we will discuss the
Cauchy problem and Gaussian estimates for the fundamental solution, the Harnack
inequality and comparison principle, and the Dirichlet problem in bounded
domains. In particular, we also justify the notion of $H$-parabolic measure and introduce the
notion of NTA domain.

\subsection{Notation}\label{SS:not}

In $\Rn$, with $n\geq 3$, we consider a system  $X =
\{X_1,...,X_m\}$ of $C^\infty$ vector fields satisfying
H\"ormander's finite rank condition \eqref{frc}. As in \cite{FP}, a piecewise $C^1$
curve $\gamma:[0,\ell]\to \Rn$ is called subunitary 
if at every $t\in [0,\ell]$ at which $\gamma'(t)$ exists one has for every $\xi\in\Rn$
\[
<\gamma'(t),\xi>^2\ \leq\ \sum_{j=1}^m <X_j(\gamma(t)),\xi>^2 .
\]
We note explicitly that the above inequality forces $\gamma '(t)$ to
belong to the span of $\{X_1(\gamma (t)),...,$ $ X_m(\gamma (t))\}$.
The subunit length of $\gamma$ is by definition $l_s(\gamma)=\ell$.
If we fix an open set $\Om\subset \Rn$, then given $x, y\in \Om$, denote by $\mathcal S_\Om(x,y)$ the collection
of all subunitary $\gamma:[0,\ell]\to \Om$ which join $x$ to $y$. The
accessibility theorem of Chow and Rashevsky, \cite{Ra}, \cite{Chow},
states that, if $\Om$ is connected, then for every
$x,y\in \Om$ there exists $\gamma \in \mathcal S_\Om(x,y)$. As a
consequence, if we define
\[
d_{\Om}(x,y) = \text{inf}\ \{l_s(\gamma)\mid \gamma \in \mathcal S_\Om(x,y)\},
\]
we obtain a distance on $\Om$, called the Carnot-Carath\'eodory distance, associated with the system $X$. When $\Om = \Rn$, we write $d(x,y)$ instead of $d_{\Rn}(x,y)$. It is clear that $d(x,y) \leq d_\Om(x,y)$, $x, y\in \Om$, for every connected open set $\Om \subset \Rn$. In \cite{NSW} it was proved that, given $\Om \subset \subset \Rn$, there exist $C, \epsilon >0$ such that
\begin{equation}\label{CCeucl}
C |x - y| \leq d_\Om(x,y) \leq C^{-1} |x - y|^\epsilon , \quad\quad\quad x, y \in \Om.
\end{equation}
This gives $d(x,y)\ \leq C^{-1} |x - y|^\epsilon$, $x, y\in \Om$, and therefore
\[
i: (\Rn, |\cdot|)\to (\Rn, d) \quad\quad\quad is\,\ continuous .
\]
Furthermore, it is easy to see that also the continuity of the opposite inclusion holds \cite{GN1}, and therefore the metric and the Euclidean topologies are equivalent.

For $x\in \Rn$ and $r>0$, we let $B_d(x,r) = \{y\in \Rn\mid d(x,y) < r \}$.
The basic properties of these balls were established by Nagel, Stein and Wainger in their seminal paper \cite{NSW}. These authors proved in particular that, given bounded open set $U\subset \Rn$, there exist constants $C, R_0>0$ such that, for any $x\in U$, and $0 < r \leq R_0$,
\begin{equation*}\label{pol}
C\le \frac{B_d(x,r)}{\Lambda(x,r)} \le C^{-1},
\end{equation*}
where  $\Lambda(x,r) = \sum_I |a_I(x)| r^{d_I}$ is a polynomial function with continuous coefficients.
As a consequence, one has with $C_1>0$, 
\begin{equation}\label{dc}
|B_d(x,2r)| \leq C_1 |B_d(x,r)| \qquad\text{for every}\quad x\in U\quad\text{and}\quad 0< r \leq R_0.
\end{equation}

In what follows, given $\beta\in (0,1)$, we let  $\Gamma^\beta(\Omega_T)$
denote the space of functions $u:\Omega_T\to \R$ such that 
\begin{align}\label{pol3}
||u||_{\Gamma^\beta(\Omega_T)}: & =\sup_{\Omega_T} |u|
\\
& +\sup_{(x,t),(x^{\prime
},t^{\prime })\in\Omega_T,\ (x,t)\neq(x^{\prime },t^{\prime })}\frac {|u(x,t)-u(x^{\prime },t^{\prime })|} {d_p(x,t,x^{\prime },t^{\prime}) ^\beta}<\infty.
\notag
\end{align}
We say that $u$ has a Lie derivative along $X_j$, at $(x,t)\in\Omega_T$, if $u\circ\gamma$ is differentiable at 0, where $\gamma$ is the integral curve of
$X_j$ such that $\gamma(0)=(x,t)$. Moreover, we indicate with $
\Gamma^{2+\beta}(\Omega_T)$ the space of functions $u\in \Gamma^\beta(\Omega_T)$ which admit Lie derivatives up to second
order along $X_1,...,X_m$, and up to order one with respect to $t$,
in $\Gamma^\beta(\Omega_T)$. If $u\in \Gamma^{2+\beta}(\Omega_T)$
then we let $||u||_{\Gamma^{2+\beta}(\Omega_T)}$ denote the naturally
defined norm of $u$. Furthermore, $u\in\Gamma^\beta_{\mbox{loc}}(\Omega_T)$
if $u\in\Gamma^\beta(D)$ for any compact subset $D$ of $\Omega_T$. The space 
$\Gamma^{2+\beta}_{\mbox{loc}}(\Omega_T)$ is defined analogously. Finally,
if $\beta=0$ then we simply write $\Gamma^{2}_{}(\Omega_T)$ for $%
\Gamma^{2+0}_{}(\Omega_T)$. Throughout the paper we will use the following
notation: 
\begin{align}  \label{pol3aa}
C_r(x,t)& =B_d(x,r)\times(t-r^2,t+r^2),  
\\
C_r^+(x,t)&= B_d(x,r)\times(t,t+r^2),
\notag\\
 C_r^-(x,t) & =B_d(x,r)\times(t-r^2,t), 
\notag \\
C_{r_1,r_2}(x,t)&=B_d(x,r_1)\times(t-r_2^2,t+r_2^2),  
\notag \\
C_{r_1,r_2}^+(x,t)&=B_d(x,r_1)\times(t,t+r_2^2),
\notag\\
C_{r_1,r_2}^-(x,t)& =B_d(x,r_1)\times(t-r_2^2,t),
\notag\end{align}
for $(x,t)\in\R^{n+1}$ and $r,r_1, r_2>0$. Furthermore, if $
\Omega\subset\Rn$ is a bounded domain, and $T>0$ and $\delta>0$ are given, then we let 
\begin{eqnarray}  \label{not1}
\Omega^\delta=\{x\in \Omega|\ d(x,\partial \Omega)>\delta\},\
\Omega_T^\delta=\Omega^\delta\times (0,T).
\end{eqnarray}

\subsection{The Cauchy problem}
Let $H$ be defined as in \eqref{Hu}, with the hypothesis \eqref{frc}, \eqref{ell} and \eqref{hol} in place. These assumptions allow us to use some basic results established in \cite{BBLU2}. In particular, for what concerns the existence of a
fundamental solution of the operator $H$, and Gaussian estimates, we will henceforth suppose, as it is done in \cite{BBLU2}, that the sub-Laplacian $\sum_{i=1}^m X_i^2$ associated with $X$ coincides with the standard Laplacian $\Delta = \sum_{j=1}^n \p_{x_j}^2$ in $\Rn$ outside of a fixed compact set in $\Rn$. 

In \cite{BBLU2} it is
proved that, under such hypothesis,  there exists a fundamental solution, $\Gamma$, for $H$, with a number of important properties. In particular, $\Gamma$ is a continuous function away from the diagonal of $\R^{n+1}\times\R^{n+1}$ and $\Gamma(x,t,\xi,\tau)=0$ for $t\leq\tau$.
Moreover, $\Gamma(\cdot,\cdot,\xi,\tau)\in \Gamma_{\mbox{loc}}^{2+\alpha}(\R^{n+1}\setminus\{(\xi,\tau)\})$ for every fixed $(\xi,\tau)\in
\R^{n+1}$ and $H(\Gamma(\cdot,\cdot,\xi,\tau))=0$ in $\R^{n+1}\setminus\{(\xi,\tau)\}$. For every $\psi\in C_0^\infty(\R^{n+1})$ the function 
\begin{equation*}
w(x,t)=\int\limits_{\R^{n+1}}\Gamma(x,t,\xi,\tau)\psi(\xi,\tau)d\xi
d\tau
\end{equation*}
belongs to $\Gamma_{\mbox{loc}}^{2+\alpha}(\R^{n+1})$ and we have $Hw=\psi$ in $\R^{n+1}$. Furthermore, let $\mu\geq 0$ and $T_2>T_1$
be such that $(T_2-T_1)\mu$ is small enough, let $0<\beta\leq\alpha$, let $%
g\in C^{0,\beta}(\R^n\times [T_1,T_2])$ and $f\in C(\R^n)$
be such that $|g(x,t)|,|f(x)|\leq c\exp(\mu d(x,0)^2)$ for some constant $%
c>0 $. Then, for $x\in\Rn, t\in
(T_1,T_2]$,
 the function 
\begin{equation}\label{rep}
u(x,t)=\int\limits_{\mathbf{R}^n}\Gamma(x,t,\xi,T_1)f(\xi)d\xi+\int%
\limits_{T_1}^t\int\limits_{\mathbf{R}^n}
\Gamma(x,t,\xi,\tau)g(\xi,\tau)d\xi d\tau,
\end{equation}
belongs to the class $\Gamma_{\mbox{loc}}^{2+\beta}(\Rn\times(T_1,T_2))\cap C(\Rn\times [T_1,T_2])$. Moreover, $u$
solves the Cauchy problem 
\begin{eqnarray}  \label{rep1}
Hu=g\mbox{ in }\R^n\times(T_1,T_2),\ u(\cdot,T_1)=f(\cdot)\mbox{ in }%
\Rn.
\end{eqnarray}
One also has the following Gaussian bounds.

\begin{lemma}\label{Gaussbound} There exist a positive
constant $C$ and, for every $T>0$, a positive constant $c=c(T)$ such that,
if $0<t-\tau\leq T$, $x,\xi\in\Rn$, then 
\begin{equation}\label{G1}
c^{-1}\frac{e^{-Cd(x,\xi)^2/(t-\tau)}}{|B(x,\sqrt{t-\tau})|}\leq\Gamma(x,t,
\xi,\tau)\leq c\frac{e^{-C^{-1}d(x,\xi)^2/(t-\tau)}}{|B(x,\sqrt{t-\tau})|}. 
\end{equation}
Furthermore, one also has
\begin{equation}\label{G2}
|X_i\Gamma(\cdot,t,\xi,\tau)(x)|\leq c(t-\tau)^{-1/2}\frac{e^{-C^{-1}d(x,\xi)^2/(t-\tau)}}{|B(x,\sqrt{t-\tau})|},  
\end{equation}
and
\begin{equation}\label{G3}
|X_iX_j\Gamma(\cdot,t,\xi,\tau)(x)|+|\partial_t\Gamma(x,\cdot,\xi,
\tau)(t)|\leq c(t-\tau)^{-1} c\frac{e^{-C^{-1}d(x,\xi)^2/(t-\tau)}}{|B(x,\sqrt{t-\tau})|}.
\end{equation}
\end{lemma}

\subsection{The Harnack inequality and strong maximum principle}

We next state the Harnack inequality and the strong maximum principle
for the operator $H$, see \cite{BBLU1} and also \cite{BBLU2}.

\begin{thrm}\label{Harnack} Let $R>0$, $0<h_1<h_2<1$ and $\gamma\in (0,1)$. Then, there exists a
positive constant $C=C(h_1,h_2,\gamma,R)$ such that the following holds for
every $(\xi,\tau)\in\R^{n+1}$, $r\in (0,R]$. If 
\begin{equation*}
u\in \Gamma^2(C_r^-(\xi,\tau))\cap C(\overline{C_r^-(\xi,\tau)})
\end{equation*}
satisfies $Hu=0$, $u\geq 0$, in $C_r^-(\xi,\tau)$, then 
\begin{equation*}
u(x,t)\leq Cu(\xi,\tau)%
\mbox{ whenever $(x,t)\in \overline{C_{\gamma r,h_2r}^-(\xi,\tau)}\setminus 
     C_{\gamma r,h_1r}^-(\xi,\tau)$}.
\end{equation*}
\end{thrm}

%\begin{theorem}\label{weakmax} (Weak maximum principle)  Let $X=\{X_1,...,X_m\}$ be a system of smooth vector fields 
%satisfying \eqref{1.1x}, \eqref{1.1xumu} and 
%assume that $A=\{a_{ij}\}$ satisfies \eqref{1.2} and \eqref{1.2+}. Furthermore, assume that $a_{ij}$ is $C^\infty$-smooth
%for all $i,j\in\{1,..,m\}$. Let 
%$H$ be defined as in \eqref{1.1}, let 
%$\Omega\subset\mathbf R^n$ be a bounded domain  and let $T>0$. Assume that there exists 
%$\hat u\in \Gamma^2(\Omega_T)$ such that $L\hat u<0$ and $\hat u>0$ in $\Omega_T$. Then the following holds. If 
%$u\in \Gamma^2(\Omega_T)$, $Lu\geq 0$ in $\Omega_T$ and $\limsup u\leq 0$ on $\partial_p\Omega_T$, then $u\leq 0$ in $\Omega_T$.
% \end{theorem}

\begin{thrm}\label{strongmax}  
Let $\Omega \subset \Rn$ be a
connected, bounded open set, and let $T>0$. 
%Suppose that locally the weak maximum principle holds for the operator $L$ in $\Omega_T$.  
Let $u\in \Gamma ^{2}(\Omega _{T})$ and assume that $Lu\geq 0$, $u\leq 0$ in 
$\Omega _{T}$. Assume that $u(x_{0},t_{0})=0$ for some $(x_{0},t_{0})\in
\Omega _{T}$. Then $u(x,t)\equiv 0$ whenever $(x,t)\in \{\Omega _{T}\cap
\{t:t\leq t_{0}\}$.
\end{thrm}

%\begin{proof} It follows fromTheorem 1.2 in \cite{BU}.\end{proof}

\subsection{The Dirichlet problem}

In the following we let $D$ be any bounded open subset of $\R^{n+1}$ and we
study the Dirichlet problem 
\begin{eqnarray}  \label{dirp}
Hu=0\mbox{ in $D$, $u=f$ on $\partial_p D$},
\end{eqnarray}
with $f\in C(\partial_pD)$. Here, $\partial_pD$
denotes the parabolic boundary of $D$. If $u:D\to\R$ is a smooth function satisfying 
$H u=0$ in $D$, then we say that $u$ is $H$-parabolic in $D $. We
denote by $P(D)$ the linear space of functions which are $H$-parabolic in $D$. 

We 
say that $D$ is $H$-regular if for any $f\in C(\partial_p D)$ there exists a unique function $H_f^D\in P(D)$ such that $\lim_{(x,t)\to
(x_0,t_0)}H_f^D(x,t)=f(x_0,t_0)$ for every $(x_0,t_0)\in\partial_p D$. Following the arguments in \cite{LU}, see in particular 
Theorems 6.5 and 10.1, we can easily construct a basis for the Euclidean topology of $\mathbb R^{n+1}$ which is made 
of cylindrical $H$-regular sets.
Furthermore, if $D$ is $H$-regular, then in view of Theorem \ref{strongmax} (one actually only needs the weak maximum principle) for every fixed $(x,t)\in D$ the map $f\mapsto H_f^D(x,t)$ defines a positive linear functional on $C(\partial_p D)$.
By the Riesz representation theorem there exists a unique Borel
measure $\omega =\omega_D$, supported in $\partial_p D$, such that 
\begin{equation}  \label{parabolic.measure.L.regular}
H_f^D(x,t)=\int_{\partial_p D}f(y,s)d\omega^{(x,t)}(y,s),\quad\text{ for every }\,f\in C(\partial_p D).
\end{equation}
We will refer to $\omega^{(x,t)}=\omega_D^{(x,t)}$ as the $H$-\emph{parabolic measure} relative to $D$
and $(x,t)$. 

A lower semi-continuous
function $u:D\to\,]-\infty,\infty]$ is said to be $H$-superparabolic
in $D$ if $u<\infty$ in a dense subset of $D$ and if 
\begin{equation*}
u(x,t)\geq\int_{\partial V}u(y,s)d\omega_V^{(x,t)}(y,s),
\end{equation*}
for every open $H$-regular set $V\subset\overline{V}\subset D$ and for every 
$(x,t)\in V$. We denote by $\overline{S}(D)$ the set of $H$-superparabolic functions in $D$, and by $\overline{S}^+(D)$ the set of the
functions in $\overline{S}(D)$ which are nonnegative. A function $v:D\rightarrow [-\infty,\infty[$ is said to be $H$-subparabolic in $D$
if $-v\in\overline{S}(D)$ and we write $\underline{S}(D):=-\overline{S}(D)$.
As the collection of $H$-regular sets is a basis for the Euclidean topology, it follows that $\overline{S}(D)\cap\underline{S}(D)=P(D)$. Finally, we recall that $H_f^D$ can be realized as
the generalized solution in the sense of Perron-Wiener-Brelot-Bauer to the
problem in \eqref{dirp}. In particular, 
\begin{equation}  \label{parabolic.measureumu}
\inf\overline{\mathcal{U}}_f^D=\sup\underline{\mathcal{U}}_f^D=H_f^D,
\end{equation}
where we have indicated with $\overline{\mathcal{U}}_f^D$ the collection of all $u\in\overline{S}(D)$ such that $\inf_{D}u>-\infty$, 
and
\[
\liminf_{(x,t)\to(x_0,t_0)}u(x,t)\geq f(x_0,t_0),
\,\forall\,(x_0,t_0)\in\partial_p D,
\]
 and with $\underline{\mathcal{U}}_f^D$ the collection of all $u\in\underline{S}(D)$ for which $\sup_{D}u<\infty$, and 
\[
\limsup_{(x,t)\to (x_0,t_0)}u(x,t)\leq f(x_0,t_0),
\,\forall\,(x_0,t_0)\in\partial_p D.
\]

\begin{lemma}\label{gendir} 
Let $D\subset\R^{n+1}$ be a bounded open set, let $f\in
C(\partial_pD)$ and let $u$ be the generalized
Perron-Wiener-Brelot-Bauer solution to the problem in \eqref{dirp}, i.e., $u=H_f^D$
where $H_f^D$ be defined as in \eqref{parabolic.measureumu}. Then $u\in
\Gamma^2(D)$.
\end{lemma}

\begin{proof}
This follows from Theorem 1.1 in \cite{U}. 
\end{proof}

In the following we are concerned with the issue of regular boundary points
and we note, concerning the solvability of the Dirichlet problem for the
operator $H$, that in \cite{U} Uguzzoni developes what he refers to as a ``cone
criterion'' for non-divergence equations modeled on H{\"o}rmander vector
fields. This is a generalization of the well-known positive density condition of classical potential theory. We next describe his result in the setting of domains of
the form $\Omega_T=\Omega\times (0,T)$, where $\Omega\subset \Rn$ is
assumed to be a bounded domain. In \cite{U}
A bounded open set $\Omega$ is said to have \emph{outer positive $d$-density at $x_0\in\partial\Omega$} if there exist $r_0$, $\theta>0$ such that 
\begin{equation}  \label{cone}
|B_d(x_0,r)\setminus\bar\Omega|\geq\theta |B_d(x_0,r)|,\ \ \text{for all}\  
r\in (0,r_0).
\end{equation}
Furthermore, if $r_0$ and $\theta$ can be chosen independently of $x_0$ then one says that $\Omega$ satisfies the outer positive $d$-density condition. The
following lemma is a special case of Theorem 4.1 in \cite{U}.

\begin{lemma}\label{Dirichlet0} 
Assume that $\Omega$ satisfies the outer positive $d$-density condition. Given $f\in C(\partial_p\Omega_T)$ and $g\in\Gamma^{\beta}(\Omega_T)$ for some $0<\beta\leq\sigma$, where $\sigma$ is the H{\"o}lder
exponent in \eqref{hol}, there exists a unique solution $u\in\Gamma^{2+\beta}(\Omega_T)\cap C(\Omega_T\cup\partial_p\Omega_T)$ to the problem 
\[
Hu=g\mbox{ in $\Omega_T$,\ \  \ \  $u=f$ on $\partial_p\Omega_T$}.
\]
In particular, $\Om_T$ is $H$-regular for the Dirichlet problem \eqref{dirp}.
\end{lemma}

\subsection{NTA domains}\label{SS:nta}

In this section we recall the notion of NTA domain with respect to the control distance $d(x,y)$ induced by the system $X = \{X_1,...,X_m\}$. We recall that, when $d(x,y) = |x-y|$, the notion of NTA domain was introduced  in \cite{JK} in connection with the study of the boundary behavior of nonnegative harmonic functions. The first study of NTA domains in a sub-Riemannian context was conducted in \cite{CG}, where a large effort was devoted to the nontrivial question of the construction of examples. In that paper the relevant Fatou theory was also developed and, in particular, the doubling condition for harmonic measure, and the comparison theorem for quotients of nonnegative solutions of sub-Laplacians. Subsequently, in the papers \cite{CGN3}, \cite{CGN4} the notion of NTA domain was combined with an intrinsic outer ball condition to obtain the complete solvability of the Dirichlet problem. 

Given a bounded open set $\Om\subset \Rn$, we recall that a ball $B_d(x,r)$ is $M$-non-tangential in $\Omega$ (with respect to the metric $d$) if 
\[
M^{-1}r<d(B_d(x,r),\partial\Omega)<Mr. 
\]
Furthermore, given $x,y\in\Omega$ a
sequence of $M$-non-tangential balls in $\Omega$, $B_d(x_1,r_1)$,..., $B_d(x_p,r_p)$, is called a Harnack chain of length $p$ joining $x$ to $y$ if 
$x\in B_d(x_1,r_1)$, $y\in B_d(x_p,r_p)$, and $B_d(x_i,r_i)\cap
B_d(x_{i+1},r_{i+1})\neq \varnothing$ for $i\in\{1,...,p-1\}$. We
note that in this definition consecutive balls have comparable radii.

\begin{dfn}\label{D:NTA}
We say that a connected, bounded open set $\Om\subset \Rn$ is a
\emph{non-tangentially accessible domain} with respect to the system $X =
\{X_1,...,X_m\}$ ($NTA$ domain, hereafter) if there exist $M$,
$r_0>0$ for which:
\begin{itemize}
\item[(i)] (Interior corkscrew condition) For any $x_0\in\partial \Om$ and $r\leq r_0$
there exists $A_r(x_0)\in \Om$ such that
$\frac{r}M<d(A_r(x_0),x_0)\leq r$ and $d(A_r(x_0),\partial
\Om)>\frac{r}M$. (This implies that $B_d(A_r(x_0),\frac{r}{2M})$ is
$(3M)$-nontangential.)
\item[(ii)] (Exterior corkscrew condition) $ \Om^c=\Rn\setminus  \Om$ satisfies
property (i).
\item[(iii)] (Harnack chain condition) There exists $C(M)>0$ such that for any $\epsilon>0$ and $x,y\in \Om$ such
that $d(x,\partial \Om)>\epsilon$, $d(y,\partial \Om)>\epsilon$, and $d(x,y)<C\epsilon$, there exists a Harnack
chain joining $x$ to $y$ whose length depends on $C$ but not on $\epsilon$.
\end{itemize}
\end{dfn}

We observe that the Chow-Rashevski accessibility theorem implies that
the metric space $(\mathbb{R}^n,d)$ be locally compact, see \cite{GN}. Furthermore, for any
bounded set $\Omega\subset\mathbb{R}^n$ there exists $R_0=R_0(\Omega)>0$
such that the closure of balls $B(x_0,R)$ with $x_0\in\Omega$ and $0<R<R_0$
are compact. We stress that metric balls of large
radii fail to be compact in general, see \cite{GN}. In view of these observations, for a
given NTA domain $\Omega\subset\mathbb{R}^n$ with constant $M$ and $r_0$
we will always assume, following \cite{CG}, that the constant $r_0$ has been
adjusted in such a way that the closure of balls $B(x_0,R)$, with $%
x_0\in\Omega$ and $0<R<r_0$, be compact.

We note the following lemma which will prove useful in the sequel
and which follows directly from Lemma \ref{Dirichlet0} and Definition \ref{D:NTA}. In its statement  the number $\sigma$ denotes the H{\"o}lder exponent in \eqref{hol}. 

\begin{lemma}\label{L:cork}
Let $\Om\subset \Rn$ be $NTA$ domain, then there exist constants $C,
R_1$, depending on the $NTA$ parameters of $\Om$, such that for every
$y\in \p \Om$ and every $0<r<R_1$ one has,
\[
C  |B_d(y,r)|\leq \min \{|\Om\cap B_d(y,r)|,|\Om^c \cap B_d(y,r)|\}
\leq C^{-1} |B_d(y,r)|.
\]
In particular, every $NTA$ domain has outer positive $d$-density and therefore, in view of Lemma \ref{Dirichlet0}, given $f\in C(\partial_p\Omega_T)$,
there exists a unique solution $u\in\Gamma^{2+\sigma}(\Omega_T)\cap
C(\Omega_T\cup\partial_p\Omega_T)$ to the Dirichlet problem \eqref{dirp}.
In particular, $\Om_T$
is $H$-regular.
\end{lemma}

Assume that $\Omega\subset\mathbb{R}^n$ is a non-tangentially accessible
domain with respect to the system $X=\{X_1,...,X_m\}$ and with parameters $%
M,r_0$. Let $T>0$ and define $\Omega_T=\Omega\times (0,T)$. Based on
Definition \ref{D:NTA}, for every $(x_0,t_0)\in S_T$, $0<r<r_0$, we introduce the following points of reference
whenever 
\begin{align}\label{points2}
{A}_r^+(x_0,t_0) & =(A_r(x_0),t_0+2r^2),
\\
A_r^-(x_0,t_0) & =(A_r(x_0),t_0-2r^2),
\notag\\
A_r(x_0,t_0) & =(A_r(x_0),t_0).
\notag
\end{align}

We note here that according to Lemma 6.4 in \cite{LU}, $\mathbb{R}^n \setminus
B_d(x_0,R)$ satisfies condition (ii) in Definition \ref{D:NTA}, and thus it
also satisfies the uniform outer positive $d$-density condition, and one can solve
the Dirichlet problem there. Also note that the same is true of the
intersection of two sets that satisfy condition (ii) in Definition \ref{D:NTA}. This is used to prove the following lemma (Theorem 6.5 in \cite{LU}) which states that one can approximate any bounded open set with a set where one can solve the Dirichlet problem  \eqref{dirp}.

\begin{lemma}
\label{regular} Let $D \subset \Rn$ be a bounded open set. Then, for
every $\delta >0$ there exists a set $D_\delta$ such that $\{x \in D:
d(x,\partial D) > \delta \} \subset D_\delta \subset D$, and $D_\delta$
satisfies the uniform outer positive $d$-density condition.
\end{lemma}

To apply the Harnack inequality to the equation \eqref{Hu} in a cylinder $\Om_T$, we will need to connect two points of $\Om_T$ with a suitable Harnack chain of parabolic cylinders. We thus introduce the relevant geometric definition.

\begin{dfn}\label{def1.5}
Let $(y_1, s_1)$, $(y_2, s_2)\in\Omega_T$, with  $s_2>s_1$ . Suppose that  $( s_2 - s_1
)^{1/2}\geq\eta^{ - 1} d(y_1,y_2)$ for some $\eta>1$, and that $%
d(y_1,\partial\Omega)>\epsilon$, $d(y_2,\partial\Omega)>\epsilon$, $%
(T-s_2)>\epsilon^2$, $s_1>\epsilon^2$ and $d_p((y_1,
s_1),(y_2,s_2))<c\epsilon$ for some $\epsilon>0$. We say that $\{C_{\hat
r_i,\hat \rho_i} ( \hat y_i, \hat s_i ) \}_{i=1}^\ell$ is a \emph{parabolic Harnack chain
of length $\ell$ connecting $(y_1, s_1)$ to $(y_2, s_2)$,} if $\hat r_i,\hat
\rho_i, \hat y_i, \hat s_i $ satisfy the following: 
\begin{itemize}
\item[(i)] $c(\eta)^{ - 1 }\leq\frac{\hat\rho_i}{\hat r_i}\leq c(\eta)$ for 
$i = 1, 2, \dots, \ell$,  
\item[(ii)] $\hat s_{ i + 1} - \hat s_i \geq c(\eta)^{ - 1 } \hat r_i^2,$
for  $i = 1, 2, \dots, \ell-1$,  
\item[(iii)] $B_d(\hat y_i,\hat r_i)$ is $M$-nontangential 
in $\Omega$  for  $i = 1, 2, \dots, \ell$, 
\item[(iv)] $(y_1, s_1)\in C_{\hat r_1,\hat \rho_1} ( \hat y_1, \hat s_1 ),\ (y_2,
s_2)\in C_{\hat r_\ell,\hat \rho_\ell} ( \hat y_\ell, \hat s_\ell )$,  
\item[(v)] $C_{\hat r_{i+1},\hat \rho_{i+1}} ( \hat y_{i+1}, \hat s_{i+1} ) \cap
C_{\hat r_i,\hat \rho_i} ( \hat y_i, \hat s_i )\neq\varnothing$ for  $i
= 1, 2, \dots, \ell-1$.
\end{itemize}
\end{dfn}

\begin{lemma}\label{lemmaacc} 
Let $\Omega\subset\Rn$
be a NTA-domain. Given  $T>0$ and $(y_1,
s_1)$, $(y_2, s_2)\in\Omega_T$, suppose that $s_2>s_1$, $( s_2 - s_1
)^{1/2}\geq\eta^{ - 1} d(y_1,y_2)$ for some $\eta>1$, that $%
d(y_1,\partial\Omega)>\epsilon$, $d(y_2,\partial\Omega)>\epsilon$, $%
(T-s_2)>\epsilon^2$, $s_1>\epsilon^2$ and that $d_p((y_1, s_1),(y_2,
s_2))<c\epsilon$ for some $\epsilon>0$. Then, there exists a parabolic
Harnack chain $\{ C_{\hat r_i,\hat \rho_i} ( \hat y_i, \hat s_i ) \}_{i=1}^\ell$,
connecting $(y_1, s_1)$ to $(y_2, s_2)$ in the sense of Definition \ref{def1.5}. Furthermore, the
length $\ell$ of the chain can be chosen to depend only on $\eta$ and $c$, but not on $\epsilon$.
\end{lemma}

\begin{proof}
Since $\Omega$ is a NTA domain and since $%
d(y_1,\partial\Omega)>\epsilon$, $d(y_2,\partial\Omega)>\epsilon$, 
\begin{equation*}
d(y_1,y_2)\leq d_p((y_1, s_1),(y_2, s_2))<c\epsilon,
\end{equation*}
it follows that we can use Definition \ref{D:NTA} to conclude the existence
of a Harnack chain of length $\hat \ell=\hat \ell(c)$, $\{ B_d(\hat y_i, \hat r_i)\}_{i=1}^{\hat \ell}$,
connecting $y_1$ and $y_2$. In the
following we let $\beta$ be a degree of freedom to be fixed below. Using $%
\beta$ we define $\hat \rho_i = \beta \hat r_i$, we let $\hat s_i = s_1 + 
\frac{1}{\beta} \sum_{j=1}^i \hat r_j^2$ for $i\in\{1,..,\hat \ell\}$, and we
consider the sequence of cylinders 
\begin{equation*}
\{C_{\hat r_i, \hat \rho_i}(\hat y_i,\hat s_i)\}_{i=1}^{\hat \ell}.
\end{equation*}
If we now choose $\beta>1$, and if we assume that $\beta$ is chosen as a
function of $\eta$, then $(i)$, $(ii)$, $(iii)$, $(v)$ and the first part of 
$(iv)$ in Definition \ref{def1.5} are satisfied. In particular, it only remains to
ensure that the second part of $(iv)$ in Definition \ref{def1.5} is satisfied. To do
this we first note that we can assume, without loss of generality, that $%
\hat r_i \leq d(y_1,y_2)$ for all $i\in\{1,...,\hat \ell\}$. Hence, $%
\sum_1^{\hat l} \hat r_i^2 \leq {\hat \ell} \cdot d(y_1,y_2)^2$. Furthermore,
since $d(y_1,y_2)^2\leq \eta^2 (s_2-s_1)$, we have 
\begin{eqnarray}  \label{Toky1}
\hat s_{\hat \ell} - s_1 = \frac{1}{\beta} \sum_1^{\hat \ell} \hat r_j^2 \leq 
\frac{\hat \ell}{\beta} d(y_1,y_2)^2 \leq \frac{\hat \ell}{\beta} \eta^2 (s_2-s_1).
\end{eqnarray}
We now let $\beta=\hat \ell \cdot \eta^2$ and we can conclude that $\hat
s_{\hat \ell} \leq s_2$. If $\hat s_{\hat \ell} = s_2$ we are done. Otherwise, we
only step up in time with cylinders $C_j = \{C_{\hat r_{\hat \ell}, \hat
r_{\hat \ell}}(y_2,\hat s_{\hat \ell}+j \hat r_{\hat \ell})\}$ until we reach $%
(y_2,s_2)$. The time that is left depends on $\eta$, and, in particular, we
have that $s_2-\hat s_{\hat \ell} \leq c^2 \epsilon^2$. Furthermore, since $%
\hat r_\ell \leq c \epsilon$, the number of steps we need to reach $(y_2,s_2)$
only depends on $c$. In particular, it is clear that the length of the
entire parabolic Harnack chain only depends on $c$ and $\eta$. 
\end{proof} 

\begin{lemma}\label{lemmaacchar} 
Let $u$ be a nonnegative solution to the equation $Hu=0$ in $\Omega_T$.
Furthermore, let $(y_1, s_1)$, $(y_2, s_2)\in\Omega_T$, suppose that $%
s_2>s_1 $, $( s_2 - s_1 )^{1/2}\geq\eta^{ - 1} d(y_1,y_2)$ for some $\eta>1$%
, that $d(y_1,\partial\Omega)>\epsilon$, $d(y_2,\partial\Omega)>\epsilon$, $%
(T-s_2)>\epsilon^2$, $s_1>\epsilon^2$ and that $d_p((y_1, s_1),(y_2,
s_2))<c\epsilon$ for some $\epsilon>0$. Then, there exists a constant $\hat
c=\hat c(H,\eta,c,r_0)$, $1\leq \hat c<\infty$, such that 
\begin{equation*}
u(y_1,s_1)\leq \hat cu(y_2,s_2).
\end{equation*}
\end{lemma}

\begin{proof} To prove the lemma we simply use the parabolic
Harnack chain from Lemma \ref{lemmaacc} and apply Theorem \ref{Harnack}
in each cylinder. Note that the dependence of constant $\hat c$ on $r_0$
enters through the size parameter $R$ in the statement of Theorem \ref%
{Harnack}. 
\end{proof}

%%%%%%%%%%%%%%%%%%%%%%%%%%%%%%%%%%%%
%>>>>>>>>>>>>>>>>>>>>>>>>>>>>>>>>>>>>>>>>>>>>>>>

%%%%%%%%%%%%%%%%%%%%%%%%%%%%%%%%%%%%

\section{Basic estimates}

The purpose of this section is to establish a number of
basic technical estimates that will be used in the proof of Theorems \ref{T:back}-\ref{T:doubling}. We mention that, using the notion of NTA domain and Lemma \ref{lemmaacchar}, several of the proofs previously established in the literature in the classical case $%
m=n$ and $\{X_1,...,X_m\}=\{\partial_{x_1},...,\partial_{x_n}\}$ can be extended to our setting. As a consequence, wherever appropriate, we will either omit details or be
brief. 
As previously, unless otherwise stated, $c $ will
denote a positive constant $\geq 1$, not necessarily the same at each
occurrence, depending only on $H$ and $M$. In general, $c ( a_1, \dots, a_m
) $ denotes a positive constant $\geq 1, $ which may depend only on $H$, $M$
and $a_1, \dots, a_m, $ and which is not necessarily the same at each
occurrence. When we write  $A\approx B$ we mean that $A/B$ is bounded from above and below by
constants which, unless otherwise stated, only depend on $H,M$.

\begin{lemma}\label{lem4.7}
Let $(x_0,t_0)\in S_T$ and  
\[
r<\min\{r_0/2,\sqrt{(T-t_0)/4},\sqrt{t_0/4}\}.
\]
Let $u$ be a nonnegative
solution to $Hu=0$ in $\Omega_T\cap C_{2r}(x_0,t_0)$ which
vanishes continuously on $\Delta(x_0,t_0,2r)$. Then, there exist $c=c(H,M,r_0)$, $1\leq c<\infty$, and $\gamma=\gamma(H,M)>0$, such that for every $(x,t)\in \Omega_T\cap C_r(x_0,t_0)$,
\begin{equation*}
u(x,t)d_p(x,t,S_T)^\gamma\leq cr^\gamma u(A^+_r(x_0,t_0)).
\end{equation*}
\end{lemma}

\begin{proof}
The proof of this lemma is based on Lemma \ref%
{lemmaacchar}. In particular, let $P_0 = (x,t) \in \Omega_T\cap C_r(x_0,t_0)$
and let $a=d_p(P_0,S_T)$. Note that, without loss of generality,
we can assume that $a<r/c_1$ for some large $c_1$ since otherwise we are done
immediately by a simple application of Lemma \ref{lemmaacchar}. Now, take $%
Q_0\in S_T$ such that $d_p(Q_0,P_0)=a$ and define $P_i=A^+_{2^i a}(Q_0)$ for
all $i\geq 1$ such that $A^+_{2^i a}(Q_0)$ is well-defined. We intend to use
Lemma \ref{lemmaacchar} to prove that $u(P_i)\leq c u(P_{i+1})$ for some
constant $c=c(H,M,r_0)$. In the following we write $P_i=(P_i^x,P_i^t)$, $%
Q_0=(Q_0^x,Q_0^t)$ to indicate the spatial and time coordinate of $P_i$ and $%
Q_0$ respectively. Then, for $i=0$ we have 
\begin{equation*}
d(P_0^x,P_1^x) \leq d(P_0^x,Q_0^x) + d(Q_0^x,P_1^x) \leq a + 2a = \frac{3}{2 
\sqrt{2}} (P_0^t-P_1^t)^{1/2}.
\end{equation*}
Since $3/2\sqrt{2}>1$, using Lemma \ref{lemmaacchar} we can conclude that $%
u(P_0)\leq c u(P_1)$. To continue, for $i\geq 1$ we first note that 
\begin{equation*}
P_{i+1}^t-P_i^t = 2(2^{i+1} a)^2 - 2(2^{i} a)^2 = 3 \cdot 2^{2i+1} a^2.
\end{equation*}
Furthermore, we also have 
\begin{equation*}
d(P_{i+1}^x,P_i^x) \leq d(P_{i+1}^x,Q_0^x) + d(Q_0^x,P_i^x) \leq 2^{i+1} a +
2^i a = \sqrt{\frac{3}{2}} (P_{i+1}^t-P_i^t)^{1/2}.
\end{equation*}
Let $\epsilon = 2^{i} a/M$. Then $d(P_i^x,\partial \Omega)>\epsilon$, $%
d(P_{i+1}^x,\partial \Omega)>\epsilon$ and $d_p(P_{i+1},P_i) = (3 \cdot
2^{2i} a^2 + 3 \cdot 2^{2i+1} a^2)^{1/2} = \sqrt{15}M\cdot \epsilon$. Since $%
\sqrt{\frac{3}{2}}$ and $\sqrt{15}M$ are both independent of $i$ and since $%
\sqrt{\frac{3}{2}}>1$, we can again conclude, using Lemma \ref{lemmaacchar},
that $u(P_i)\leq Cu(P_{i+1})$ for all $i>0$ such that $P_i$ and $P_{i+1}$
lie in $\Omega_T\cap C_{2r}(x_0,t_0)$. In particular, to complete the proof
it is now enough to consider the largest $k$ such that $2^{k} a \leq r$ and
then iterate the above inequalities in a standard fashion. We omit further
details.
\end{proof}

\begin{lemma}
\label{lem4.8} Let $(x_0,t_0)\in S_T$ and  
\[
r<\min\{r_0/2,\sqrt{(T-t_0)/4},\sqrt{t_0/4}\}.
\]
Let $u$ be a nonnegative
solution to $Hu=0$ in $\Omega_T\cap C_{2r}(x_0,t_0)$ vanishing continuously on $\Delta(x_0,t_0,2r)$. Then, there exist $c=c(H,M,r_0)$, $1\leq c<\infty$, and $\gamma=\gamma(H,M,r_0)>0$, such that 
\begin{equation*}
u(A^-_r(x_0,t_0))\leq c\left(\frac{r}{d_p(x,t,\partial_p\Omega_T)}\right)^\gamma u(x,t),
\end{equation*}
whenever $(x,t)\in \Omega_T\cap C_r(x_0,t_0)$.
\end{lemma}

\begin{proof}
To prove this lemma one can proceed similarly to the
proof of Lemma \ref{lem4.7}. 
\end{proof} 

\begin{lemma}\label{lem4.5-Kyoto1} 
There exists a $\hat{K}\gg 1$, $\hat{K}=\hat{K}(H,M)$, such that the following is
true whenever $(x_{0},t_{0})\in \R^{n+1}$ and $r<r_{0}/(2\hat{K})$.
Assume that $D$ is a domain in $\Rn$ such that $D\subset
B_{d}(x_{0},\hat{K}r)$ and assume that there exist $\hat{x}_{0}\in
B_{d}(x_{0},\hat{K}r)$ and $\rho>0 $ such that $B_{d}(\hat{x}_{0},2\rho
)\subset B_{d}(x_{0},r)$, $B_{d}(\hat{x}_{0},2\rho )\cap D=\varnothing $ and $%
M^{-1}r<\rho <r$. Let $u$ be a function in $D\times (t_{0}-4r^{2},t_{0})$
which satisfies $Hu\geq 0$ in $D\times (t_{0}-4r^{2},t_{0})$, $u\leq 0$ on $%
\partial _{p}(D\times (t_{0}-4r^{2},t_{0}))\setminus \partial _{p}C_{\hat{K}%
r,2r}^{-}(x_{0},t_{0})$ and $\sup_{D\times (t_{0}-4r^{2},t_{0})}u>0$. Then,
there exists a constant $\theta =\theta (H,M,r_{0})$, $0<\theta <1$, such
that 
\begin{equation}
\sup_{(D\times (t_{0}-4r^{2},t_{0}))\cap C_{r}^{-}(x_{0},t_{0})}u\leq \theta
\sup_{D\times (t_{0}-4r^{2},t_{0})}u.  \label{jul2}
\end{equation}
\end{lemma}

\begin{proof} Let $\hat{K}\gg 1$ be a constant to be
fixed below. We let $\phi _{1}\in C_{0}^{\infty }(\Rn)$ be such
that $0\leq \phi _{1}\leq 1$, $\phi _{1}\equiv 1$ on $B_{d}(x_{0},\hat{K}%
r+r)\setminus B_{d}(x_{0},\hat{K}r-r)$, $\phi _{1}\equiv 0$ on $B_{d}(x_{0},%
\hat{K}r-2r)\cup (\Rn\setminus B_{d}(x_{0},\hat{K}r+2r))$.
Similarly, we let $\phi _{2}\in C_{0}^{\infty }(\Rn)$ be such
that $0\leq \phi _{2}\leq 1$, $\phi _{2}\equiv 1$ on $B_{d}(x_{0},\hat{K}%
r)\setminus B_{d}(\hat{x}_{0},2\rho )$, $\phi _{2}\equiv 0$ on $B_{d}(\hat{x}%
_{0},\rho )\cup (\Rn\setminus B_{d}(x_{0},\hat{K}r+2r))$. Using $%
\phi _{1}$ and $\phi _{2}$ we define 
\begin{eqnarray*}
\Phi _{1}(\hat{x},\hat{t}) &=&\int\limits_{\Rn}\Gamma (\hat{x},%
\hat{t},\xi ,t_{0}-4r^{2})\phi _{1}(\xi )d\xi ,  \notag  \label{compp5again}
\\
\Phi _{2}(\hat{x},\hat{t}) &=&\int\limits_{\Rn}\Gamma (\hat{x},%
\hat{t},\xi ,t_{0}-4r^{2})\phi _{2}(\xi )d\xi ,
\end{eqnarray*}%
whenever $(\hat{x},\hat{t})\in \R^{n+1}$, $\hat{t}\geq t_{0}-4r^{2}$%
. To preceed we first prove that there exist a constant $c$ such that%
\begin{equation}
1\leq c\Phi _{1}(\hat{x},\hat{t})\ \text{for}\ (\hat{x},\hat{t})\in
\partial _{p}(C_{\hat{K}r,2r}^{-}(x_{0},t_{0})\cap
\{(x,t):t_{0}-4r^{2}<t<t_{0}\})  \label{jul1}
\end{equation}%
To establish this, let $(\hat{x},\hat{t})$ be as in (\ref{jul1}), and for
simplicity assume that $t_{0}-4r^{2}=0$. Then, using Lemma \ref{Gaussbound} and (\ref{dc})
we see that%
\begin{eqnarray*}
\Phi _{1}(\hat{x},\hat{t}) &\geq &\int_{B_{d}(\hat{x},\sqrt{\hat{t}}%
/2)}\Gamma (\hat{x},\hat{t},\xi ,0)\phi _{1}(\xi )d\xi  \notag \\
&\geq &\int_{B_{d}(\hat{x},\sqrt{\hat{t}}/2)}c^{-1}\left\vert B(\hat{x},%
\sqrt{\hat{t}})\right\vert ^{-1}e^{-Cd(\hat{x},\xi )^{2}/\hat{t}}d\xi  \notag
\\
&=&e^{-C\hat{t}/4\hat{t}}\int_{B_{d}(\hat{x},\sqrt{\hat{t}}%
/2)}c^{-1}\left\vert B(\hat{x},\sqrt{\hat{t}})\right\vert ^{-1}e^{-C\left(
4d(\hat{x},\xi )^{2}-\hat{t}\right) /4\hat{t}}d\xi  \notag \\
&\geq &e^{-C/4}c^{-1}\left\vert B(\hat{x},\sqrt{\hat{t}})\right\vert
^{-1}\int_{B_{d}(\hat{x},\sqrt{\hat{t}}/2)}d\xi \geq e^{-C/4}c^{-1}\hat{C}^{-1}.  \label{compky5+}
\end{eqnarray*}%
We conclude that (\ref{jul1}) holds provided that we choose $c\le e^{-C/4}\hat{C}^{-1}$. Now, let%
\begin{equation}
M=\sup_{D\times (t_{0}-4r^{2},t_{0})}u.  \label{supM}
\end{equation}%
Using (\ref{jul1}) and the maximum principle on $D\times
(t_{0}-4r^{2},t_{0})$ we thus see that the estimate
\begin{equation}
u(\hat{x},\hat{t})\leq cM\Phi _{1}(\hat{x},\hat{t})+M\Phi _{2}(\hat{x},\hat{t%
})  \label{compky4}
\end{equation}%
holds in $D\times (t_{0}-4r^{2},t_{0})$, and thus in particular in $\left( D\times
(t_{0}-4r^{2},t_{0}\right) )\cap C_{r}^{-}(x_{0},t_{0})$. Further, if $(\hat{%
x},\hat{t})\in \left( D\times (t_{0}-4r^{2},t_{0}\right) )\cap
C_{r}^{-}(x_{0},t_{0})$, then%
\begin{eqnarray*}
\Phi _{1}(\hat{x},\hat{t}) &\leq &\int\limits_{B_{d}(x_{0},\hat{K}%
r+r)\setminus B_{d}(x_{0},\hat{K}r-r)}|B(\hat{x},\sqrt{\hat{t}}%
)|^{-1}e^{-C^{-1}d(\hat{x},\xi )^{2}/\hat{t}}d\xi  \notag \\
&\leq &\int\limits_{B_{d}(x_{0},\hat{K}r+r)\setminus B_{d}(x_{0},\hat{K}%
r-r)}|B(\hat{x},r)|^{-1}e^{-c^{-1}d(\hat{x},\xi )^{2}/r^{2}}d\xi  \notag \\
&\leq &ce^{-c^{-1}\hat{K}^{2}}|B_{d}(x_{0},\hat{K}r+r)\setminus B_{d}(x_{0},%
\hat{K}r-r)||B(\hat{x},r)|^{-1}  \notag \\
&\leq &ce^{-c^{-1}\hat{K}^{2}}|B_{d}(x_{0},\hat{K}r+r)||B(\hat{x},r)|^{-1}.
\end{eqnarray*}%
Iterating (\ref{dc}) and using that $r<r_{0}/(2\hat{K})$ we see that 
\begin{equation*}
ce^{-c^{-1}\hat{K}^{2}}|B_{d}(x_{0},\hat{K}r+r)||B(\hat{x},r)|^{-1}\leq
ce^{-c^{-1}\hat{K}^{2}}\hat{K}_{{}}^{\eta }
\end{equation*}%
for some integer $\eta >>1$ which is independent of $\hat{K}$, $x_{0}$, $%
\hat{x}$ and $r$. In particular%
\begin{equation*}
\Phi _{1}(\hat{x},\hat{t})\leq ce^{-c^{-1}\hat{K}^{2}}\hat{K}^{\eta }.
\end{equation*}%
To estimate $\Phi _{2}(\hat{x},\hat{t})$ we note that%
\begin{eqnarray*}
\Phi _{2}(\hat{x},\hat{t}) &=&1-\hat{\Phi}_{2}(\hat{x},\hat{t}),\text{ where}
\\
\hat{\Phi}_{2}(\hat{x},\hat{t}) &=&\int_{\Rn}\Gamma (\hat{x},\hat{%
t},\xi ,t_{0}-4r^{2})(1-\phi _{2}(\xi ))d\xi,
\end{eqnarray*}%
and by construction,%
\begin{equation*}
\hat{\Phi}_{2}(\hat{x},\hat{t})\geq \int_{B_{d}(\hat{x},\rho )}\Gamma (\hat{x%
},\hat{t},\xi ,t_{0}-4r^{2})d\xi .
\end{equation*}%
As before we then prove that 
\begin{equation*}
\hat{\Phi}_{2}(\hat{x},\hat{t})\geq c^{-1},
\end{equation*}%
and actually, for $\varepsilon $ small enough,%
\begin{equation*}
\hat{\Phi}_{2}(\hat{x}_{0},t_{0}-4r^{2}+\varepsilon ^{2}\rho ^{2})\geq
c^{-1}.
\end{equation*}%
Hence, by using the Harnack inequality we can conclude that $\Phi _{2}(\hat{x%
},\hat{t})=1-\hat{\Phi}_{2}(\hat{x},\hat{t})\leq \left( 1-c^{-1}\right)$,
whenever $(\hat{x},\hat{t})\in \left( D\times (t_{0}-4r^{2},t_{0}\right)
)\cap C_{r}^{-}(x_{0},t_{0})$ for some $c=c(H,M,r_{0})>1.$ In particular, for every $(\hat{x},\hat{t})\in \left(
D\times (t_{0}-4r^{2},t_{0}\right) )\cap C_{r}^{-}(x_{0},t_{0})$, we have
\begin{equation*}
u(\hat{x},\hat{t})\leq cM\Phi _{1}(\hat{x},\hat{t})+M\Phi _{2}(\hat{x},\hat{t%
})\leq M(ce^{-c^{-1}\hat{K}^{2}}\hat{K}^{\eta }+\left( 1-\hat{c}^{-1}\right)
),
\end{equation*}
for some $\hat{c}=\hat{c}(H,M,r_{0})$.  Given $%
\hat{c}$, we choose $\hat{K}$ so that $ce^{-c^{-1}\hat{K}^{2}}\hat{K}^{\eta
}\leq \hat{c}^{-1}/2$, and we let $\theta =\left( 1-\hat{c}^{-1}/2\right) <1$%
. Then, the following inequality holds
\begin{equation}
u(\hat{x},\hat{t})\leq \theta M,  \label{compky10b}
\end{equation}%
 with $M$ as in (\ref{supM}). This establishes (\ref{jul2}), thus
completing the proof. 
\end{proof}

We will also need a few variations on the theme of Lemma \ref{lem4.5-Kyoto1}.

\begin{cor}\label{lem4.5-} 
There exists a $\hat{K}\gg 1$, $\hat{K}=\hat{K}(H,M,r_{0})$, such that the following is true whenever $%
(x_{0},t_{0})\in S_{T}$ and 
\[
r<\min \{r_{0}/(2\hat{K}),\sqrt{(T-t_{0})/4},
\sqrt{t_{0}/4}\}.
\]
Let $u$ be a non-negative solution to $Hu=0$ in $\Omega
_{T}\cap C_{\hat{K}r,2r}^{-}(x_{0},t_{0})$ vanishing
continuously on $S_{T}\cap C_{\hat{K}r,2r}^{-}(x_{0},t_{0})$. Then, there
exists a constant $\theta =\theta (H,M)$, $0<\theta <1$, such that 
\begin{equation*}
\sup_{\Omega _{T}\cap C_{r}^{-}(x_{0},t_{0})}u\leq \theta \sup_{\Omega
_{T}\cap C_{\hat{K}r,2r}^{-}(x_{0},t_{0})}u.
\end{equation*}
\end{cor}

\begin{proof} This is an obvious consequence of the NTA character of $\Omega$ and of Lemma \ref{lem4.5-Kyoto1}. We omit further details.
\end{proof}

\begin{lemma}\label{lem4.5--} 
There exists a $\hat{K}\gg 1$, $\hat{K}%
=\hat{K}(H,M,r_{0})$, such that the following is true whenever $%
(x_{0},t_{0})\in S_{T}$ and \[
r<\min \left\{r_{0}/(2\hat{K}),\sqrt{(T-t_{0})/(4%
\hat{K})^{2}},\sqrt{t_{0}/(4\hat{K})^{2}}\right\}. 
\]
Let $u$ be a solution to $Hu=0$
in $\Omega _{T}\cap C_{\hat{K}r,2r}^{-}(x_{0},t_{0})$ which
vanishes continuously on $S_{T}\cap C_{\hat{K}r,2r}^{-}(x_{0},t_{0})$. Then,
there exists a constant $\theta =\theta (H,M)$, $0<\theta <1$, such that 
\begin{equation*}
\sup_{\Omega _{T}\cap C_{r}^{-}(x_{0},t_{0})}u^{\pm }\leq \theta
\sup_{\Omega _{T}\cap C_{\hat{K}r,2r}^{-}(x_{0},t_{0})}u^{\pm },
\end{equation*}%
where $u^{+}(x,t)=\max \{0,u(x,t)\}$, $u^{-}(x,t)=-\min \{0,u(x,t)\}$.
\end{lemma}

\begin{proof}
We first prove Lemma \ref{lem4.5--} for $u^{+}$.
In fact, in this case the argument is essentially the same as that in the proof of
Lemma \ref{lem4.5-Kyoto1}. In particular, if we let 
\begin{equation*}
M^{+}=\sup_{\Omega _{T}\cap C_{\hat{K}r,2r}^{-}(x_{0},t_{0})}u^{+},
\end{equation*}%
then we see that (\ref{compky4}) still holds but with $M$ replaced by $M^{+}$%
. Furthermore, repeating the argument in \eqref{compky5+} - \eqref{compky10b},
we see that 
\begin{equation*}
u(\hat{x},\hat{t})\leq \theta M^{+},
\end{equation*}%
whenever $(\hat{x},\hat{t})\in \Omega _{T}\cap C_{r}^{-}(x_{0},t_{0})$.
Obviously this completes the proof of Lemma \ref{lem4.5--} for $u^{+}$.
Concerning the same estimate for $u^{-}$ we see, by analogy, that 
\begin{equation}
-u(\hat{x},\hat{t})\leq \theta M^{-},\ M^{-}=\sup_{\Omega _{T}\cap C_{\hat{K}%
r,2r}^{-}(x_{0},t_{0})}(-u)=\sup_{\Omega _{T}\cap C_{\hat{K}%
r,2r}^{-}(x_{0},t_{0})}u^{-},  \label{compky10jajb}
\end{equation}%
whenever $(\hat{x},\hat{t})\in \Omega _{T}\cap C_{r}^{-}(x_{0},t_{0})$ and
from (\ref{compky10jajb}) we deduce Lemma \ref{lem4.5--} for $u^{-}$. This
completes the proof of the lemma.
\end{proof}

\begin{lemma}\label{lem4.5} 
Let $(x_{0},t_{0})\in S_{T}$ and let $%
r<\min \{r_{0}/2,\sqrt{(T-t_{0})/4},\sqrt{t_{0}/4}\}$. Let $u$ be a
non-negative solution to $Hu=0$ in $\Omega _{T}\cap C_{2r}(x_{0},t_{0})$ which vanishes continuously on $\Delta (x_{0},t_{0},2r)$. Then,
there exist a constant $c=c(H,M,r_{0})$, $1\leq c<\infty $, and $\alpha
=\alpha (H,M)\in (0,1)$, such that 
\begin{equation}
u(x,t)\leq c\biggl (\frac{d_{p}(x,t,x_{0},t_{0})}{r}\biggr )^{\alpha
}\sup_{\Omega _{T}\cap C_{2r}(x_{0},t_{0})}u
\end{equation}%
whenever $(x,t)\in \Omega _{T}\cap C_{r/c}(x_{0},t_{0})$.
\end{lemma}

\begin{proof} This lemma is a simple consequence of Corollary \ref{lem4.5-}. 
\end{proof} 

\begin{lemma}\label{lem4.6} 
Let $(x_0,t_0)\in S_T$ and let $%
r<\min\{r_0/2,\sqrt{(T-t_0)/4},\sqrt{t_0/4}\}$. Let $u$ be a nonnegative
solution to $Hu=0$ in $\Omega_T\cap C_{2r}(x_0,t_0)$ vanishing continuously on $\Delta(x_0,t_0,2r)$. Then, there exists a constant $%
c=c(H,M,r_0)$, $1\leq c<\infty$, such that 
\begin{equation*}
u(x,t)\leq cu(A_r^+(x_0,t_0))
\end{equation*}
whenever $(x,t)\in\Omega_T\cap C_{r/c}(x_0,t_0)$.
\end{lemma}

\begin{proof} This lemma is a consequence of Lemma \ref{lem4.5},
the Harnack inequality and a classical argument developed in [CFMS] and
[Sa]. 
\end{proof} 

\begin{rmrk}
\label{rem4.6} Note that if $u$ is a nonnegative solution to $Hu=0$ in all
of $\Omega_T$ then Lemma \ref{lem4.6} can be improved in the following way.
Let $(x_0,t_0)\in S_T$ and $r$ be as in the statement
of Lemma \ref{lem4.6}. Let $u$ be a nonnegative solution to $Hu=0$ in $%
\Omega_T$ vanishing continuously on $\Delta(x_0,t_0,2r)$.
Then, there exists a constant $c=c(H,M,r_0)$, $1\leq c<\infty$, such that 
\begin{equation*}
u(x,t)\leq cu(A_r^+(x_0,t_0))
\end{equation*}
whenever $(x,t)\in\Omega_T\cap C_{r}(x_0,t_0)$. In fact, the restriction $%
(x,t)\in\Omega_T\cap C_{r/c}(x_0,t_0)$ in Lemma \ref{lem4.6} is simply a
result of the fact that we in Lemma \ref{lem4.6} are only assuming that $u$
is a nonnegative solution in $\Omega_T\cap C_{2r}(x_0,t_0)$.
\end{rmrk}

\begin{lemma}
\label{lem4.9} Let $u$ be a nonnegative solution to $%
Hu=0$ in $\Omega_T$ which vanishes continuously on $S_T$. Let $%
0<\delta\ll\sqrt{T}$ be given. Then, there exists a constant $%
c=c(H,M,\mbox{diam}(\Omega),T,\delta,r_0)$, $1\leq c<\infty$, such that 
\begin{equation*}
\sup_{(x,t)\in\Omega^\delta\times (\delta^2,T)}u(x,t)\leq c
\inf_{(x,t)\in\Omega^\delta\times(\delta^2,T)}u(x,t).
\end{equation*}
\end{lemma}

\begin{proof} To prove this we can proceed, using the lemmas
given above, exactly as in the proof of Lemma 2.7 in \cite{N}. 
\end{proof}

\begin{lemma}\label{lem4.95} 
Let $K\gg1$ be given, let 
$(x_0,t_0)\in S_T$ and assume that $r<\min\{r_0/(8K),\sqrt{(T-t_0)/64},\sqrt{%
t_0/64}\}$. Let $u$ be a nonnegative solution to the equation $Hu=0$ in $%
\Omega_T$ vanishing continuously on $S_T$. Let $%
\gamma=\gamma(H,M)\in (0,1)$ be as in Lemma \ref{lem4.7} and Lemma \ref%
{lem4.8}. Assume that 
\begin{eqnarray*}
\sup_{\Omega_T\cap C_{2Kr,2r}^-(x_0,t_0)} u\geq (2K)^{-\gamma}
\sup_{\Omega_T\cap C_{4Kr,8r}^-(x_0,t_0)} u.
\end{eqnarray*}
Then, provided $K=K(H,M)$ is chosen large enough, there exists $c=c(H,M,r_0)\geq 1$, such that 
\begin{eqnarray*}
\sup_{\Omega_T\cap C_{4Kr}^-(x_0,t_0)\cap\{(x,t): t=t_0-64r^2\}} u\geq
c^{-1}\sup_{\Omega_T\cap C_{2Kr,2r}^-(x_0,t_0)} u.
\end{eqnarray*}
\end{lemma}

\begin{proof}
The proof of this lemma is similar to that of
Lemma \ref{lem4.5-}. In particular, we let $\phi_1\in C_0^\infty(\Rn)$ be such that $0\leq\phi_1\leq 1$, $\phi_1\equiv 1$ on $%
B_d(x_0,4Kr+2r)\setminus B_d(x_0,4Kr-2r)$, $\phi_1\equiv 0$ on $%
B_d(x_0,4Kr-4r)\cup(\Rn\setminus B_d(x_0,4Kr+4r))$. Since $\Omega$
is NTA we see that there exist $\hat x_0$ and $\rho>0$ such that $%
r/M<4\rho<r$ and such that $B(\hat x_0,2\rho)\subset(\Rn\setminus
\Omega)\cap B(x_0,r)$. Based on this we let $\phi_2\in C_0^\infty(\Rn)$ be such that $0\leq\phi_2\leq 1$, $\phi_2\equiv 1$ on $%
B_d(x_0,4Kr)\setminus B(\hat x_0,2\rho)$, $\phi_2\equiv 0$ on $(\Rn\setminus B_d(x_0,4 Kr+4r))\cup B(\hat x_0,\rho)$. Using $\phi_1$ and $%
\phi_2$ we define 
\begin{eqnarray*}  \label{compp5b}
\Phi_1(\hat x,\hat t)&=&\int\limits_{\Rn}\Gamma(\hat x,\hat
t,\xi,t_0-64r^2)\phi_1(\xi)d\xi,  \notag \\
\Phi_2(\hat x,\hat t)&=&\int\limits_{\Rn}\Gamma(\hat x,\hat
t,\xi,t_0-64r^2)\phi_2(\xi)d\xi,
\end{eqnarray*}
whenever $(\hat x,\hat t)\in\R^{n+1}$, $\hat t\geq t_0-64r^2$. Let $%
\Gamma_1=\Omega_T\cap C^-_{4Kr}(x_0,t_0)\cap\{(x,t): t=t_0-64r^2\}$, $%
\Gamma_2= \partial_p(\Omega_T\cap
C^-_{4Kr,8r}(x_0,t_0))\setminus\Gamma_1\setminus S_T$. In the following we
let 
\begin{eqnarray*}  \label{compp2}
M= \sup_{\Omega_T\cap C^-_{4Kr,8r}(x_0,t_0)} u,\ \hat M=\sup_{\Omega_T\cap
C^-_{4Kr}(x_0,t_0)\cap\{(x,t): t=t_0-64r^2\}} u.
\end{eqnarray*}
Then, by arguing as in the proof of Lemma \ref{lem4.5-}, we first see that
there exists $c$ such that 
\begin{equation*}  \label{compp5ky+0}
1\leq c\Phi_1(\hat x,\hat t)\ \text{for}\ (\hat x,\hat t)\in
\partial_p(C^{-}_{4K r,8r}(x_0,t_0))\cap \{(x,t):t_0-64r^2<t<t_0\},
\end{equation*}
and then, by the maximum principle we see, that 
\begin{eqnarray*}  \label{compp1ny1}
u(\hat x,\hat t)\leq cM\Phi_1(\hat x,\hat t)+\hat M\Phi_2(\hat x,\hat t)
\end{eqnarray*}
for $(\hat x,\hat t)\in \Omega_T\cap C^-_{2Kr,2r}(x_0,t_0)$. As in the
proof of Lemma \ref{lem4.5-} we can then deduce that 
\begin{eqnarray*}  \label{compp1ny2}
u(\hat x,\hat t)\leq cMe^{-c^{-1}{\ K}^2}{K}^\eta+\hat M\Phi_2(\hat x,\hat t),
\end{eqnarray*}
for $(\hat x,\hat t)\in \Omega_T\cap C^-_{2Kr,2r}(x_0,t_0)$. Next, using the
assumption stated in the lemma we see that 
\begin{eqnarray*}  \label{compp1ny3}
(2K)^{-\gamma}M\leq cMe^{-c^{-1}{\ K}^2}{K}^\eta+\hat M\sup_{\Omega_T\cap
C_{2Kr,2r}^-(x_0,t_0)}\Phi_2(\hat x,\hat t).
\end{eqnarray*}
Hence, assuming that $K$ is so large that $(2K)^{-\gamma}>ce^{-c^{-1}{\ K%
}^2}{K}^\eta$, we have that 
\begin{equation*}  \label{compp1ny4}
((2K)^{-\gamma}-e^{-c^{-1}{\ K}^2}{K}^\eta)M\leq\hat M\sup_{\Omega_T\cap
C_{2Kr,2r}^-(x_0,t_0)}\Phi_2(\hat x,\hat t)\leq\hat M.
\end{equation*}
In particular, we can conclude, for $K=K(H,M)$ large enough, that 
\begin{eqnarray*}  \label{compp1ny5}
\frac 1 2 (2K)^{-\gamma}\sup_{\Omega_T\cap C_{2Kr,2r}^-(x_0,t_0)} u\leq\frac
1 2 (2K)^{-\gamma}M\leq\hat M.
\end{eqnarray*}
This completes the proof.
\end{proof}

\begin{lemma}
\label{lem4.10} Let $\hat K$ be as in the statement of
Lemma \ref{lem4.5-Kyoto1}, let $K\gg\hat K$ be a constant to be suitably chosen, $(x_0,t_0)\in S_T$ and assume 
\[
r<\min\{r_0/(2K\hat K),\sqrt{(T-t_0)/(4K^2)},
\sqrt{t_0/(4K^2)}\}.
\]
 Let $u$ be a solution to $Hu=0$ in $(\Omega_T\setminus
\Omega_T^r)\cap C_{Kr}^-(x_0,t_0)$ which is continuous on the closure of $%
(\Omega_T\setminus \Omega_T^r)\cap C_{Kr}^-(x_0,t_0)$. Moreover, assume that 
\begin{eqnarray*}
(i)&&u(x,t)\leq 1\mbox{ whenever }(x,t)\in (\Omega_T\setminus
\Omega_T^r)\cap C_{Kr}^-(x_0,t_0), \\
(ii)&&u(x,t)\leq 0\mbox{ whenever }(x,t)\in
[(\partial\Omega\cup\partial\Omega^r) \times(t_0-(Kr)^2,t_0)]\cap
C_{Kr}^-(x_0,t_0).
\end{eqnarray*}
Then, there exists a constant $c=c(H,M,r_0)$, $1\leq c<\infty$, such that 
\begin{equation*}
u(x,t)\leq e^{-cK}
\end{equation*}
whenever $(x,t)\in (\Omega_T\setminus\Omega_T^r)\cap C_{\hat K r}^-(x_0,t_0)$%
.
\end{lemma}

\begin{proof}In the following we consider odd integers $2j+1$
where $j\in [0,(K/\hat K-1)/2]$. For each such $j$ we define a point $(\hat
X_j,\hat t_j)\in (\Omega_T\setminus \Omega_T^r)\cap C_{(2j+1)\hat
Kr}^-(x_0,t_0)$ through the relation 
\begin{eqnarray}  \label{sapp1}
\sup_{(\Omega_T\setminus \Omega_T^r)\cap C_{(2j+1)\hat Kr}^-(x_0,t_0)}
u=u(\hat X_j,\hat t_j).
\end{eqnarray}
We then note, using the maximum principle, that $(\hat X_j,\hat t_j)\in
\partial_p[(\Omega_T\setminus \Omega_T^r)\cap C_{(2j+1) \hat Kr}^-(x_0,t_0)]$%
. By construction we also see that there exists $(\tilde X_j,\hat t_j)\in
\partial\Omega\times[t_0-((2j+1)\hat K r)^2,t_0)$ such that $d_p(\tilde
X_j,\hat t_j, \hat X_j,\hat t_j)= d(\tilde X_j, \hat X_j)\leq r$. In
particular, $(\hat X_j,\hat t_j)$ is in the closure of $C^-_{r}(\tilde
X_j,\hat t_j)$. We next note that 
\begin{equation}  \label{sapp1-}
C^-_{\hat Kr,2r}(\tilde X_j,\hat t_j)\cap [(\Omega_T\setminus
\Omega_T^r)\cap C_{Kr}^-(x_0,t_0)]\subset (\Omega_T\setminus \Omega_T^r)\cap
C_{(2j+3)\hat Kr}^-(x_0,t_0).
\end{equation}
Let $D$ be defined through the relation $D\times(\hat t_j-4r^2,\hat
t_j)=C^-_{\hat Kr,2r}(\tilde X_j,\hat t_j)\cap [(\Omega_T\setminus
\Omega_T^r)\cap C_{Kr}^-(x_0,t_0)]$. Then, applying Lemma \ref{lem4.5-Kyoto1}
we see that there exists $\theta=\theta(H,M)$, $0<\theta<1$, such that 
\begin{eqnarray}  \label{sapp1-+}
\sup_{(D\times(\hat t_j-4r^2,\hat t_j))\cap C_{r}^-(\tilde X_j,\hat
t_j)}u\leq\theta \sup_{D\times(\hat t_j-4r^2,\hat t_j)}u.
\end{eqnarray}
In particular, since $(\hat X_j,\hat t_j)$ is in the closure of the set $%
(D\times(\hat t_j-4r^2,\hat t_j))\cap C_{r}^-(\tilde X_j,\hat t_j)$ we can
use continuity of $u$, \eqref{sapp1-+} and \eqref{sapp1-} to conclude that 
\begin{align}\label{sapp2}
u(\hat X_j,\hat t_j)& \leq\theta\sup_{D\times(\hat t_j-4r^2,\hat
t_j)}u
\\
& \leq\theta\sup_{(\Omega_T\setminus \Omega_T^r)\cap C_{(2j+3)\hat
Kr}^-(x_0,t_0)} u=\theta u(\hat X_{j+1},\hat t_{j+1}).
\notag\end{align}
Let $j_0$ be the largest positive integer such that $(2j_0+3)\hat K\leq K$.
Then, by iteration we see that 
\begin{eqnarray}  \label{sapp3}
\sup_{(\Omega_T\setminus \Omega_T^r)\cap C_{\hat Kr}^-(x_0,t_0)} u=u(\hat
X_1,\hat t_1)\leq \theta^{j_0}u(\hat X_{j_0+1},\hat t_{j_0+1})\leq
\theta^{j_0}
\end{eqnarray}
where we, at the last step, has used that $u(x,t)\leq 1$. Hence 
\begin{eqnarray}  \label{sapp3+}
\sup_{(\Omega_T\setminus \Omega_T^r)\cap C_{\hat K r}^-(x_0,t_0)}
u\leq\theta^{j_0}.
\end{eqnarray}
Obviously \eqref{sapp3+} implies the statement in Lemma \ref{lem4.10} and
the proof is complete. 
\end{proof}

\begin{lemma}\label{lem4.11} 
Let $\hat K$ be as in the statement of
Lemma \ref{lem4.5-Kyoto1}, let $K\gg\hat K$ be given, $(x_0,t_0)\in S_T$ and assume that \[
r<\min\{r_0/(2K),\sqrt{(T-t_0)/(4K^2)},
\sqrt{t_0/(4K^2)}\}.
\]
Let $u$ and $v$ be two solutions to $Hu=0$ in $(\Omega_T\setminus \Omega_T^r)\cap C_{Kr}^-(x_0,t_0)$. Moreover, assume that 
\begin{eqnarray*}
(i)&&u(x,t)\geq 0,\ v(x,t)\leq 1 \mbox{ whenever }(x,t)\in
(\Omega_T\setminus \Omega_T^r)\cap C_{Kr}^-(x_0,t_0), \\
(ii)&&u(x,t)\geq 1\mbox{ whenever }(x,t)\in [\partial\Omega^r
\times(t_0-(Kr)^2,t_0)]\cap C_{Kr}^-(x_0,t_0), \\
(iii)&&v(x,t)\leq 0\mbox{ whenever }(x,t)\in
[(\partial\Omega\cup\partial\Omega^r) \times(t_0-(Kr)^2,t_0)]\cap
C_{Kr}^-(x_0,t_0).
\end{eqnarray*}
Then, for any $(x,t)\in \Omega_T\cap C_r^-(x_0,t_0)$ one has 
\[
v(x,t)\leq u(x,t),
\]
provided $K=K(H,M)$ is chosen large enough.
\end{lemma}

\begin{proof} To start the proof of Lemma \ref{lem4.11} we claim
that if $u$ as in the statement of the lemma, then 
\begin{equation}
u(x,t)\geq 2\epsilon \biggl (\frac{d_{p}(x,t,S_{T})}{r}\biggr )^{\eta }%
\mbox{ whenever }(x,t)\in \Omega _{T}\cap C_{r}^{-}(x_{0},t_{0}),
\label{sapp4}
\end{equation}%
where $\epsilon $ and $\eta $ are positive constants depending only on $H,M$%
. However, we postpone the proof of this claim until the end.
We thus establish the lemma assuming \eqref{sapp4}. To do this we first
note that \eqref{sapp4} implies that 
\begin{equation}
u(x,t)\geq 2\epsilon K^{-\eta }\mbox{ whenever }(x,t)\in \Omega
_{T}^{r/K}\cap C_{r}^{-}(x_{0},t_{0}).  \label{sapp4+}
\end{equation}%
Furthermore, since $v$ satisfies the assumptions stated in Lemma \ref{lem4.10}, from this result we see that 
\begin{equation}
v(x,t)\leq e^{-cK}\leq \epsilon K^{-\eta }\mbox{ whenever }(x,t)\in \Omega
_{T}\cap C_{r}^{-}(x_{0},t_{0}),  \label{sapp5}
\end{equation}%
provided $K=K(H,M)$ is large enough. In particular, 
\begin{equation*}
\mbox{$v(x,t)\leq \epsilon K^{-\eta} \leq u(x,t)$ whenever
$(x,t)\in \Omega_T^{r/K}\cap C_r^-(x_0,t_0)$}
\end{equation*}%
We now define for $(x,t)\in (\Omega _{T}\setminus \Omega _{T}^{r/K})\cap
C_{r}^{-}(x_{0},t_{0})$, 
\begin{equation*}
u_{1}(x,t)=\frac{K^{\eta }}{2\epsilon }u(x,t),\ v_{1}(x,t)=\frac{K^{\eta }}{%
2\epsilon }(2v(x,t)-u(x,t)).
\end{equation*}%
Then, using \eqref{sapp4+}, \eqref{sapp5}, we see that 
\begin{eqnarray*}
(i_{1}) &&u_{1}(x,t)\geq 0,\ v_{1}(x,t)\leq 1\mbox{ whenever }(x,t)\in
(\Omega _{T}\setminus \Omega _{T}^{r/K})\cap C_{r}^{-}(x_{0},t_{0}), \\
(ii_{1}) &&u_{1}(x,t)\geq 1\mbox{ whenever }(x,t)\in \lbrack \partial \Omega
^{r/K}\times (t_{0}-(r)^{2},t_{0})]\cap C_{r}^{-}(x_{0},t_{0}), \\
(iii_{1}) &&v_{1}(x,t)\leq 0\mbox{ whenever }(x,t)\in \lbrack (\partial
\Omega \cup \partial \Omega ^{r/K})\times (t_{0}-r^{2},t_{0})]\cap
C_{r}^{-}(x_{0},t_{0}).
\end{eqnarray*}%
Moreover, $u_{1}$, $v_{1}$ are solutions to $Hu=0$ in $(\Omega
_{T}\setminus \Omega _{T}^{r/K})\cap C_{r}^{-}(x_{0},t_{0})$. In particular,
the pair $(u_{1},v_{1})$ satisfies the assumptions stated in Lemma \ref%
{lem4.11} with $r$ replaced by $r/K$. Furthermore, by construction we have
that 
\begin{equation*}
u(x,t)-v(x,t)=\frac{\epsilon }{K^{\eta }}(u_{1}(x,t)-v_{1}(x,t))\geq 0,
\end{equation*}%
whenever $(x,t)\in \Omega _{T}^{r/K^{2}}\cap C_{r/K}^{-}(x_{0},t_{0})$.
Hence, by iteration of this argument we see that we can construct functions $%
u_{j}$ and $v_{j}$, for $j=1,2...$, such that 
\begin{equation*}
u(x,t)-v(x,t)=\biggl (\frac{\epsilon }{K^{\eta }}\biggr )%
^{j}(u_{j}(x,t)-v_{j}(x,t))\geq 0
\end{equation*}%
whenever $(x,t)\in (\Omega _{T}\setminus \Omega _{T}^{r/K^{j+1}})\cap
C_{r/K^{j}}^{-}(x_{0},t_{0})$. As a consequence we obtain that 
\begin{equation*}
u(x,t)-v(x,t)\geq 0\mbox{ whenever }(x,t)\in I(x_{0},t_{0}),
\end{equation*}%
where $I(x_{0},t_{0})=\bigcup_{j=1}^{\infty }\Omega _{T}^{r/K^{j}}\cap
C_{r/K^{j-1}}^{-}(x_{0},t_{0})$. Finally, for arbitrary $(\hat{x}_{0},\hat{t}%
_{0})\in \Omega _{T}\cap C_{r}^{-}(x_{0},t_{0})$ one can choose $(\tilde{x}%
_{0},\tilde{t}_{0})\in S_{T}$ such that $d_{p}(\hat{x}_{0},\hat{t}%
_{0},S_{T})=d_{p}(\hat{x}_{0},\hat{t}_{0},\tilde{x}_{0},\tilde{t}_{0})=d(%
\hat{x}_{0},\tilde{x}_{0})$. Then $(\hat{x}_{0},\hat{t}_{0})\in I(\tilde{x}%
_{0},\tilde{t}_{0})$ and $d(x_{0},\tilde{x}_{0})<r$, i.e., 
\begin{align*}
& (\Omega _{T}\setminus \Omega _{T}^{r})\cap C_{Kr}^{-}(\tilde{x}_{0},\tilde{%
t}_{0})\subset (\Omega _{T}\setminus \Omega _{T}^{r})\cap
C_{(K+2)r}^{-}(x_{0},t_{0}),  \notag  \label{sapp10} \\
&\partial _{p}\bigl ((\Omega _{T}\setminus \Omega _{T}^{r})\cap C_{Kr}^{-}(%
\tilde{x}_{0},\tilde{t}_{0})\bigr )\cap \partial _{p}\Omega _{T}^{r}\subset
\partial _{p}\bigl ((\Omega _{T}\setminus \Omega _{T}^{r})\cap
C_{(K+2)r}^{-}(x_{0},t_{0})\bigr )\cap \partial _{p}\Omega _{T}^{r}.
\end{align*}%
Hence, by replacing $K$ with $K+2$ in the original assumptions and repeating
the proof up to here with $(\tilde{x}_{0},\tilde{t}_{0})$ instead of $%
(x_{0},t_{0})$, we can conclude that $u(x,t)-v(x,t)\geq 0$ on $I(\tilde{x}%
_{0},\tilde{t}_{0})$ and in particular, $u(\hat{x}_{0},\hat{t}_{0})-v(\hat{x}%
_{0},\hat{t}_{0})\geq 0$. Since $(\hat{x}_{0},\hat{t}_{0})\in \Omega
_{T}\cap C_{r}^{-}(x_{0},t_{0})$ is arbitrary we can hence conclude that $%
u-v\geq 0$ on $\Omega _{T}\cap C_{r}^{-}(x_{0},t_{0})$. In particular, to
complete the proof of Lemma \ref{lem4.11} we are only left with proving the claim
in \eqref{sapp4}.

To do this we proceed as follows. Let $\hat K$ be
as in the statement of Lemma \ref{lem4.5-Kyoto1} and let $\Lambda\gg 1$ be given. Assume that $K\gg \Lambda \hat K$. Given $x_0\in
\partial\Omega$, according to Lemma \ref{regular} we can find a set $U$ such
that 
\[
B_d(x_0,\Lambda \hat K r) \subset U \subset B_d(x_0,(\Lambda +1)\hat K
r),
\]
and such that we can solve the Dirichlet problem \eqref{dirp} in 
\[
\Omega_T \cap [U
\times (t_0-16t^2,t_0)].
\]
Furthermore, we choose $\tilde x_0\in\Omega$ and $%
\Lambda$ so that $\tilde x_0\in\partial B_d(x_0,\Lambda\hat K r)$ and $%
B_d(\tilde x_0,2\hat K r)\subset\Omega$. We note that, since $\Omega$ is an NTA domain, this can always be
accomplished by choosing $\Lambda$ large
enough. We next introduce an auxiliary function $\tilde u$ as follows. We let 
$\tilde u$ be such that $H\tilde u=0$ in $\Omega_T\cap [U \times
(t_0-16t^2,t_0)]$, $\tilde u=1$ on $\partial_p(\Omega_T\cap [U \times
(t_0-16t^2,t_0)])\cap C^-_{\hat K r,2r}(\tilde x_0,t_0-4r^2)$ and $\tilde
u=0 $ on the rest of $\partial_p(\Omega_T\cap [U \times (t_0-16t^2,t_0)])$.
We then have $0\leq\tilde u\leq 1$, and $\tilde u\leq u$ where $u$ and $%
\tilde u$ are both defined. Also, $\tilde u$ is not identical to 1 in $%
\Omega_T\cap [U \times (t_0-16t^2,t_0)]$.

Let $D=U\cap B_d(\tilde x_0,\hat K r)$ and define $\hat u=1-\tilde u$ in $%
D\times(t_0-8r^2,t_0-4r^2)$. Then $\hat u$ satisfies $H\hat u=0$ in $%
D\times(t_0-8r^2,t_0-4r^2)$, $\hat u\leq 0$ on $\partial_p(D%
\times(t_0-8r^2,t_0-4r^2))\setminus\partial_pC_{\hat Kr,2r}^-(\tilde
x_0,t_0-4r^2)$ and $\sup_{D\times(t_0-8r^2,t_0-4r^2)} \hat u>0$. Because of
the construction of $U$, there exists $\hat x_0\in B_d(\tilde x_0,\hat Kr)$
and $\rho$ such that $B_d(\hat x_0,\rho)\subset B_d(\tilde x_0,r)$, $%
B_d(\hat x_0,\rho)\cap D=\varnothing$ and $\hat M^{-1}r<\rho<r$ for some $\hat
M$ independent of $r$. We can now apply Lemma \ref{lem4.5-Kyoto1} to
conclude that there exists a constant $\theta$, $0<\theta<1$, independent of 
$r$, such that 
\begin{eqnarray}  \label{lex1}
\sup_{(D\times(t_0-8r^2,t_0-4r^2))\cap C_{r}^-(\tilde x_0,t_0-4r^2)}\hat
u\leq\theta \sup_{D\times(t_0-8r^2,t_0-4r^2)}\hat u\leq \theta.
\end{eqnarray}
In particular, by continuity we see from \eqref{lex1} that 
\begin{eqnarray}  \label{lex2}
u(\tilde x_0,t_0-4r^2)\geq\tilde u(\tilde x_0,t_0-4r^2)\geq 1-\theta>0.
\end{eqnarray}
Furthermore, using \eqref{lex2}, the Harnack inequality and Lemma \ref%
{lem4.8} we see that 
\begin{eqnarray}  \label{lex3}
1-\theta\leq cu(A^-_r(x_0,t_0))\leq c^2r^\gamma u(x,t)d_p(x,t,S_T)^{-\gamma}
\end{eqnarray}
whenever $(x,t)\in \Omega_T\cap C_r^-(x_0,t_0)$. Obviously this gives %
\eqref{sapp4} with $\eta=\gamma$ and $2\epsilon=(1-\theta)/c^2$. This
completes the proof.
\end{proof}

\section{Proof of Theorem \ref{T:back} and Theorem \ref{T:quotients}}\label{S:back}

The purpose of this section is proving Theorems \ref{T:back} and \ref{T:quotients}.

\subsection{Proof of Theorem \ref{T:back}}

To begin the proof we let $0<\delta\ll \sqrt{T}$ be a fixed
constant, we let $(x_0,t_0)\in S_T$, $\delta^2\leq t_0\leq T-\delta^2$, and
we assume that $r<\min\{r_0/2,\sqrt{(T-t_0-\delta^2)/4},\sqrt{%
(t_0-\delta^2)/4}\}$. For $\hat r>0$ we define 
\begin{eqnarray}  \label{pTh1eq1}
f(\hat r)=\hat r^{-\gamma}\sup_{\Omega_T\cap C_{2\hat r}^-(x_0,t_0)}u(x,t)
\end{eqnarray}
where $\gamma$ is the constant appearing in Lemma \ref{lem4.7}. Furthermore,
we let 
\begin{eqnarray}  \label{pTh1eq2}
\rho=\max\{\hat r:\ r\leq\hat r\leq\delta,\ f(\hat r)\geq f(r)\}.
\end{eqnarray}
By the definition of $\rho$ in \eqref{pTh1eq2} we see that 
\begin{eqnarray}  \label{pTh1eq3}
\sup_{\Omega_T\cap C_{2r}^-(x_0,t_0)}u(x,t)\leq (r/\rho)^\gamma
\sup_{\Omega_T\cap C_{2\rho}^-(x_0,t_0)}u(x,t).
\end{eqnarray}
Furthermore, using Lemma \ref{lem4.8} we see that 
\begin{eqnarray}  \label{pTh1eq4}
u(A^-_{2\rho}(x_0,t_0))\leq c(\rho/r)^\gamma u(A^-_{r}(x_0,t_0)).
\end{eqnarray}
In the following we prove that 
\begin{eqnarray}  \label{pTh1eq5}
\sup_{\Omega_T\cap C_{2\rho}^-(x_0,t_0)}u(x,t)\leq cu(A^-_{2\rho}(x_0,t_0))
\end{eqnarray}
for this particular choice of $\rho$. In fact, combining \eqref{pTh1eq3}, %
\eqref{pTh1eq4} and \eqref{pTh1eq5} we see that 
\begin{eqnarray}  \label{pTh1eq3ky}
\sup_{\Omega_T\cap C_{2r}^-(x_0,t_0)}u(x,t)\leq cu(A^-_{r}(x_0,t_0)).
\end{eqnarray}
To prove \eqref{pTh1eq5} we let $K\gg1$ be given as in Lemma \ref{lem4.95}, and we divide the proof into two cases. First, we
assume that $\delta/(2K)<\rho$. In this case $\rho$ is large and combining
Lemma \ref{lem4.6} and Lemma \ref{lem4.9} we see that 
\begin{eqnarray}  \label{pTh1eq6}
\sup_{\Omega_T\cap C_{2\rho}^-(x_0,t_0)}u(x,t)\leq
cu(A^+_{2\rho}(x_0,t_0))\leq c^2u(A^-_{2\rho}(x_0,t_0)),
\end{eqnarray}
for some $c=c(H,M,\mbox{diam}(\Omega),T,\delta,K)$, $1\leq c<\infty$. Hence,
the proof is complete in this case. Second, we assume that $r\leq\rho\leq
\delta/(2K)$ and we then first note, by the definition of $\rho$, that $%
f(2K\rho)\leq f(\rho)$, i.e., 
\begin{eqnarray*}  \label{pTh1eq7}
\sup_{\Omega_T\cap C_{2\rho}^-(x_0,t_0)}u\geq
(2K)^{-\gamma}\sup_{\Omega_T\cap C_{4K\rho}^-(x_0,t_0)}u.
\end{eqnarray*}
Obviously \eqref{pTh1eq7} implies 
\begin{eqnarray*}  \label{pTh1eq8}
\sup_{\Omega_T\cap C_{2K\rho,2\rho}^-(x_0,t_0)} u\geq (2K)^{-\gamma}
\sup_{\Omega_T\cap C_{4K\rho,8\rho}^-(x_0,t_0)} u,
\end{eqnarray*}
and hence we can use Lemma \ref{lem4.95} to conclude that 
\begin{eqnarray}  \label{pTh1eq9}
\sup_{\Omega_T\cap C_{4K\rho}^-(x_0,t_0)\cap\{(x,t): t=t_0-64\rho^2\}} u\geq
c^{-1}\sup_{\Omega_T\cap C_{2K\rho,2\rho}^-(x_0,t_0)} u.
\end{eqnarray}
In particular, using if necessary Lemma \ref{lem4.6}, and the Harnack
inequality in Theorem \ref{Harnack}, we can now use \eqref{pTh1eq9} to conclude \eqref{pTh1eq5}. This
completes the proof of \eqref{pTh1eq5}. Furthermore, Theorem 1 now follows
readily from \eqref{pTh1eq5}.

\hfill $\Box $

\subsection{Proof of Theorem \ref{T:quotients}} To prove Theorem \ref{T:quotients} we first establish a few lemmas.

\begin{lemma}\label{lem4.12} 
Let $K\gg1$ be the constant appearing
in Lemma \ref{lem4.11}, let $(x_0,t_0)\in S_T$ and assume that 
\[
r<\min\{r_0/(2K),\sqrt{(T-t_0)/(4K^2)},\sqrt{t_0/(4K^2)}\}.
\]
Let $u$ and $v$
be two nonnegative solutions to $Hu=0$ in $\Omega_T$, and assume that $v=0$
continuously on $\Delta(x_0,t_0,2Kr)$. Then, there exists a constant $c =
c(H,M,r_0)$ such that 
\begin{equation*}
\sup_{\Omega_T\cap C_{r}^-(x_0,t_0)}\frac v u\leq c\frac {%
v(A^+_{Kr}(x_0,t_0))}{u(A^-_{Kr}(x_0,t_0))}.
\end{equation*}
\end{lemma}

\begin{proof} We first note that if we choose $K$ large enough
then, since $(\Omega_T\setminus \Omega_T^r)\cap
C_{Kr}^-(x_0,t_0)\subset \Omega_T\cap C_{Kr}^-(x_0,t_0)$, we can use Remark \ref{rem4.6} to conclude that 
\begin{equation}  \label{vav1}
v(x,t)\leq c_1v(A^+_{Kr}(x_0,t_0))\mbox{ whenever }(x,t)\in
(\Omega_T\setminus \Omega_T^r)\cap C_{Kr}^-(x_0,t_0).
\end{equation}
Furthermore, by the Harnack inequality we have that 
\begin{equation*}  \label{vav2}
u(x,t)\geq c_2^{-1}u(A^-_{Kr}(x_0,t_0)),
\end{equation*}
for every $(x,t)\in \partial_p\bigl ((\Omega_T\setminus \Omega_T^r)\cap C_{Kr}^-(x_0,t_0)\bigr )
\cap\partial_p\Omega_T^r$.
For $(x,t)\in (\Omega_T\setminus \Omega_T^r)\cap C_{Kr}^-(x_0,t_0)$ let 
\[
\tilde v(x,t)=v(x,t)/v(A^+_{Kr}(x_0,t_0)),
\]
\[
\tilde u(x,t)=u(x,t)/u(A^-_{Kr}(x_0,t_0)),
\]
\[
\hat v(x,t)= c_1^{-1}\tilde
v(x,t)-c_2\tilde u(x,t),
\]
and 
\[
\hat u(x,t)=c_2\tilde u(x,t).
\] 
Then, we can
apply Lemma \ref{lem4.11} with $u,v$ replaced by $\hat u,\hat v$ to first
conclude that $\hat v(x,t)\leq \hat u(x,t)$, for $(x,t)\in \Omega_T\cap
C_{r}^-(x_0,t_0)$, and then that 
\begin{align*}  \label{vav3}
\frac {v(x,t)}{u(x,t)} &\leq \frac {v(A^+_{Kr}(x_0,t_0))}{u(A^-_{Kr}(x_0,t_0))}%
\frac {c_2}{c_1}\biggl (\frac {\hat v(x,t)}{\hat u(x,t)}+1\biggr )
\\
& \leq c_3 
\frac {v(A^+_{Kr}(x_0,t_0))}{u(A^-_{Kr}(x_0,t_0))}
\notag\end{align*}
whenever $(x,t)\in \Omega_T\cap C_{r}^-(x_0,t_0)$. This completes the proof of
Lemma \ref{lem4.12}. 
\end{proof}

\begin{lemma}\label{lem4.13} 
Let $K\gg1$ be the constant appearing
in Lemma \ref{lem4.11}, let $(x_0,t_0)\in S_T$ and assume that 
\[
r<\min\{r_0/(2K),\sqrt{(T-t_0)/(4K^2)},\sqrt{t_0/(4K^2)}\}.
\]
 Let $u$ and $v$
be two nonnegative solutions to $Hu=0$ in $\Omega_T$, assume that $u=0$
continuously on $S_T$, that $v=0$ continuously on $\Delta(x_0,t_0,4Kr)$, and
that $u$ and $v$ are not identically zero. Then, the quotient $v/u$ is H{\"o}%
lder continuous on the closure of $\Omega_T\cap C_{r}^-(x_0,t_0)$.
\end{lemma}

\begin{proof}
To prove this lemma we proceed similarly to \cite{FSY}.
Given $(x,t)$ in the closure of $\Omega_T$ and $\rho>0$ we define 
\begin{eqnarray}  \label{osc}
\omega(x,t,\rho)=\sup_{\Omega_T\cap C_{\rho}^-(x,t)}\frac v u-
\inf_{\Omega_T\cap C_{\rho}^-(x,t)}\frac v u.
\end{eqnarray}
Then, to start with, we note that Lemma \ref{lem4.12} implies that 
\begin{eqnarray}  \label{osc1}
\omega(x_0,t_0,2r)\leq 2c\frac {v(A^+_{Kr}(x_0,t_0))}{u(A^-_{Kr}(x_0,t_0))}%
\leq C<\infty.
\end{eqnarray}
In the following we let $(x,t)$ be an arbitrary point in $\Omega_T\cap
C_{r}^-(x_0,t_0)$ and we consider $0<\rho\leq r$. Let $d:=d(x,\partial%
\Omega)=d_p(x,t,S_T)$. We divide the proof into the cases $\rho\leq d$ and $%
\rho> d$.\newline

\noindent The case $\rho\leq d$. Assume first that, in addition, $\rho\leq
d/2$. We note that we can assume, without loss of generality, that 
\begin{itemize}  \label{osc2-}
\item[(i)] $0\leq \frac{v(y,s)}{u(y,s)}\leq 1$, for $(y,s)\in
C_\rho^{-}(x,t)$,
\item[(ii)] $\omega(x,t,\rho)=1$,
\item[(iii)] $\frac{v(x,t-\rho^2/2)}{u(x,t-\rho^2/2)}\geq \frac 1 2$.
\end{itemize}
To see this notice that to achieve (i) and (ii) we can replace $v$ by 
\begin{displaymath}
\hat{v}\equiv\omega(x,t,\rho)^{-1}\left(v-\left(\inf_{\Omega_T\cap C_{\rho}^-(x,t)}v/u\right)u\right).
\end{displaymath}
Furthermore, if (iii) does not hold, then we can replace $v$ by $\bar{v}\equiv u-\hat{v}\geq 0$ to achieve (iii).
Next, using the Harnack inequality we first see that 
\begin{eqnarray*}  \label{osc2}
v(x,t-\rho^2/2)\leq c v(y,s),\ u(y,s)\leq cu(x,t+\rho^2/2)
\end{eqnarray*}
whenever $(y,s)\in C_{\rho/2}^-(x,t)$. Moreover, as in the proof of
Theorem 1, we derive that 
\begin{eqnarray*}  \label{osc3}
u(x,t+\rho^2/2)\leq cu(x,t-\rho^2/2).
\end{eqnarray*}
Thus 
\begin{eqnarray*}  \label{osc5}
\frac 1 2\leq \frac{v(x,t-\rho^2/2)}{u(x,t-\rho^2/2)}\leq c\frac{v(y,s)}{%
u(y,s)}\leq c,
\end{eqnarray*}
whenever $(y,s)\in C_{\rho/2}^{-}(x,t)$, and hence 
\begin{eqnarray}  \label{osc6}
\omega(x,t,\rho/2)\leq\tilde\theta_1\omega(x,t,\rho)
\end{eqnarray}
where $\tilde\theta_1=1-1/(2c)\in (0,1)$. Furthermore, iterating the
estimate in \eqref{osc6} we deduce that 
\begin{eqnarray}  \label{osc7}
\omega(x,t,\rho)\leq\biggl (\frac {2\rho}d\biggr )^{\sigma_1}\omega(x,t,d)
\end{eqnarray}
whenever $\rho\leq d/2$ and where $\sigma_1=-\log_2\tilde \theta_1$.
Obviously this estimate also holds whenever $d/2<\rho\leq d$. \newline

\noindent The case $\rho> d$. Assume first, in addition, that $\rho<r/2$.
Note that $C_\rho^-(x,t)\subset C_{2\rho}^-(\tilde x_0,t)$ for some $\tilde
x_0\in\partial\Omega$ such that $d=d(x,\tilde x_0)$. Then by arguing as in
the proof in the case $\rho\leq d$, using Lemma \ref{lem4.12}, it follows as
in \cite{FSY} that 
\begin{align}  \label{osc2-+}
\omega(x,t,\rho)& \leq \omega(\tilde x_0,t,2\rho)\leq \biggl (\frac {4K\rho}r%
\biggr )^{\sigma_2}\omega(\tilde x_0,t,r)
\\
& \leq \biggl (\frac {4K\rho}r\biggr
)^{\sigma_2} \omega(x_0,t,2r),
\notag\end{align}
for some $\sigma_2\in (0,1)$. Obviously \eqref{osc2-+} also holds in the
case $r/2\leq \rho\leq r$.\newline

\noindent Combining \eqref{osc7} and \eqref{osc2-+} we see that 
\begin{equation*}  \label{osc6++}
\omega(x,t,\rho)\leq\biggl (\frac {2\rho}d\biggr )^{\sigma_1}\omega(x,t,d)%
\leq \biggl (\frac {2\rho}d\biggr )^{\sigma_1}\biggl (\frac {4Kd}r\biggr )%
^{\sigma_2}\omega(x_0,t,2r)
\end{equation*}
also when $\rho\leq d<r/2$. Finally, using (\ref{osc1}) we obtain for some $\sigma_3\in (0,1)$
\begin{eqnarray*}  \label{osc6+++}
\omega(x,t,\rho)\leq cK\biggl (\frac {\rho}r\biggr )^{\sigma_3}%
\omega(x,t,2r) \leq cK\biggl (\frac {\rho}r\biggr )^{\sigma_3} C,
\end{eqnarray*}
 whenever $\rho\leq d<r$. Combining these
estimates completes the proof of the lemma. 
\end{proof}

\begin{proof}[Proof of Theorem \ref{T:quotients}] 
The interior case is straightforward
since both $u$ and $v$ are H\"{o}lder continuous and since we only consider
solutions which are nonnegative and not identically zero. Hence, we have that
the quotient $v/u$ is H\"{o}lder continuous in $\Omega _{T}^{\prime }\subset
\subset \Omega _{T}\cap C_{r}^{-}(x_{0},t_{0})$. Then to prove Theorem \ref{T:quotients} we
first assume that $u=0$ continuously on $S_{T}$. In this case, using
Lemma \ref{lem4.13}, we see that $v/u$ is H{\"{o}}lder continuous on the closure of 
$\Omega _{T}\cap C_{r_{1}}^{-}(\hat{x},\hat{t})$, for some small $r_{1}>0$,
whenever $(\hat{x},\hat{t})\in S_{T}\cap C_{r}^{-}(x_{0},t_{0})$. Combining
this fact with the interior argument we see that $v/u$ is H{\"{o}}lder
continuous on the closure of $\Omega _{T}\cap C_{r}^{-}(x_{0},t_{0})$. In
the general case we represent $u$ in the form $u=u_{0}+u_{1}$, where $%
Lu=Lu_{1}=0$ in $\Omega _{T}$, $u_{0}=0$, $u_{1}=u$ on $S_{T}$ and $u_{0}=u$%
, $u_{1}=0$ on $\Omega \times \{t=0\}$. Then by the argument above we see
that $v/u_{0}$ as well as $u_{1}/u_{0}$ are H{\"{o}}lder continuous on the
closure of $\Omega _{T}\cap C_{r}^{-}(x_{0},t_{0})$. Using this we can
conclude, as 
\begin{equation*}
\frac{v}{u}=\frac{v}{u_{0}}\frac{1}{1+\frac{u_{1}}{u_{0}}},
\end{equation*}%
that also $v/u$ is H{\"{o}}lder continuous on the closure of $\Omega
_{T}\cap C_{r}^{-}(x_{0},t_{0})$. This completes the proof of Theorem \ref{T:quotients}.
\end{proof}

\section{Proof of Theorem \ref{T:doubling}}\label{S:doubling}

The objective of this section is proving Theorem \ref{T:doubling}. With this in mind, we first
need to introduce some additional notation. In particular, for $(x_0,t_0)\in\R^{n+1}$, $0<r_1<r_2$, and $K\gg 100$, we define,
\begin{align}  \label{2.6}
\Gamma_K^+(x_0,t_0,r_1,r_2) =  & \{(x,t)\mid d(x,x_0)\leq K | t - t_0 |^{1/2},
\\
& r_1\leq | t - t_0 |^{1/2}\leq r_2,\ t>t_0\}.
\notag\end{align}
Furthermore, given $(x_0,t_0)\in S_T$, $0<\rho$, $\mu\in (0,1)$, and a
function $u$, we let 
\begin{eqnarray}  \label{2.6+}
f_1^u(x_0,t_0,\rho,\mu)=\inf_{\{(x,t):\ x\in\Omega^{\mu\rho^{\prime }}\cap
B_d(x_0,\rho),\ t=t_0+\rho^2\}} u(x,t)
\end{eqnarray}
where $\rho^{\prime }=\min\{\rho, r_0\}$. Similarly, given $(x_0,t_0)\in S_T$
and a function $u$ we define 
\begin{equation}  \label{2.6++}
f_2^u(x_0,t_0,\rho,K)=\sup_{\{(x,t)\in\Omega_T\cap \partial_p
C_{K\rho,\rho}(x_0,t_0)\cap \{(x,t):\ |t-t_0|<\rho^2\}\}} u^-(x,t),
\end{equation}
where $u^-(x,t)=-\min\{0,u(x,t)\}$. 

To establish Theorem \ref{T:doubling} we will first prove
four lemmas.

\begin{lemma}\label{lem4.5-ky0}  
Let $\omega^{(x,t)}$
be the $H$-parabolic measure at $(x,t)\in\Omega_T$. Let $(x_0,t_0)\in \partial_p \Omega_T$ and assume that $r<\min\{r_0/2\}$. Then,
there exists a constant $c=c(H,M,r_0)$, $1\leq c<\infty$, such that 
\begin{equation*}
\omega^{(x,t)}(C_{2r}(x_0,t_0) \cap \partial_p \Omega_T)\geq c^{-1},
\end{equation*}
whenever $(x,t)\in \Omega_T\cap C_r(x_0,t_0)$.
\end{lemma}

\begin{proof}
First, let $(x_0,t_0) \in
S_T$. By Lemma \ref{regular}, we can choose a set $U$ which is regular for
the Dirichlet problem and such that \[
B_d(x_0,3r/2) \subset U \subset
B_d(x_0,2r).\]
Since $\Omega$ is NTA, there exists a point $A'_r(x_0)\in \mathbb R^n \setminus \Omega$ such that
\[
\frac{r}{M} < d(A'_r(x_0),x_0) \le r,\ \ \ \text{and}\ \ d(A'_r(x_0),\partial \Omega) >\frac{r}{M}.
\]
Furthermore, using Lemma \ref{regular} once again we can
also find a set $U^{\prime }$,  which is $H$-regular for the Dirichlet problem, such that 
\[
B_d(A'_r(x_0),r/4M) \subset
U^{\prime }\subset B_d(A'_r(x_0),r/2M) \subset U \setminus \Omega.
\]
Using this notation we let 
\[
C:=
U\times [t_0 - 4r^2, t_0+4r^2],\ \  C^{\prime }:= U^{\prime }\times [t_0 -
4r^2, t_0+4r^2],
\]
and $B = U^{\prime }\times \{t_0-4r^2\}$. We also let $v(x,t) = \omega_C^{(x,t)}(B)$ and $v^{\prime }(x,t) =
\omega_{C^{\prime }}^{(x,t)}(B)$. By the maximum principle, we have $%
\omega^{(x,t)}(\Delta(x_0,t_0,2r)) \geq v(x,t)$ in $C_r(x_0,t_0) \cap \Omega_T$%
, and $v(x,t) \geq \tilde v(x,t)$ in $C^{\prime }$. By the Harnack principle
applied in $C$, we have 
\begin{equation*}
\inf_{C_r(x_0,t_0) \cap \Omega_T} \omega^{(x,t)}(\Delta(x_0,t_0,2r) )\geq
\inf_{C_r(x_0,t_0) \cap \Omega_T} v(x,t) \geq
\end{equation*}
\begin{equation*}
c^{-1} v(A'_r(x_0), t_0-2r^2) \geq c^{-1} v^{\prime }(A'_r(x_0), t_0-2r^2).
\end{equation*}
We can extend the function $v^{\prime }$ to the cylinder $C^{\prime \prime
}=U' \times [t_0 - 5r^2, t_0+4r^2]$ by setting 
\begin{equation*}
v'(x,t) = \omega_{C^{\prime \prime }}^{(x,t)}(\partial_p (C^{\prime \prime }\cap \{t\leq
t_0-4r^2\})),
\end{equation*}
that is, letting $v^{\prime }=1$ below $B$. We now apply the Harnack
inequality to $v^{\prime }$ in $C^{\prime \prime }$ and obtain 
\begin{equation*}
v^{\prime }(A'_r(x_0), t_0-2r^2) \geq c^{-1} v^{\prime }(A'_r(x_0), t_0-4r^2)
= c^{-1},
\end{equation*}
and we are finished. The case when $(x_0,t_0)$ is on the bottom of $%
\partial_p \Omega_T$ is similar, but simpler. 
\end{proof}

\begin{lemma}
\label{lema1} Let $K\gg1$ be given, $(x_0,t_0)\in S_T$, and assume that $u$ be a solution to $Hu=0$ in $\Omega_T$
such that $u\geq 0$ in $\Omega_T\cap \Gamma_K^+(x_0,t_0,\rho_0,R)$ for some $%
\rho_0$ and $R$ such that $0<2\rho_0\leq R\leq\nu r_0$, where $\nu>0$ is a fixed
constant. Then, for every $\mu\in (0,1)$ there exists a $\gamma_1>0$
depending on $H$, $\mu$, $K$, $\nu$, and $r_0$, such that 
\begin{equation*}
\inf_{\rho_0\leq r\leq 2\rho_0}f_1^u(x_0,t_0,r,\mu)\leq \biggl (\frac \rho
{\rho_0}\biggr )^{\gamma_1}f_1^u(x_0,t_0,\rho,\mu)
\end{equation*}
for all $\rho$ such that $0<2\rho_0\leq\rho\leq R$.
\end{lemma}

\begin{proof}
To prove Lemma \ref{lema1} we
let $\rho$ satisfy $0<2\rho_0\leq\rho\leq R$. We define 
\begin{equation*}
D(x_0,t_0,\rho,\mu):=\{(x,t):\ x\in\Omega^{\mu\rho ^{\prime }}\cap
B_d(x_0,\rho),\ t=t_0+\rho^2\},
\end{equation*}
and note that 
%Then, given $\rho$ and $\epsilon>0$ sufficiently small, we first note that the sets
% \begin{eqnarray}\label{sets}
% D_1(x_0,t_0,\rho,\mu)&:=&\{(x,t):\ x\in\Omega^{\mu\rho'/2}\cap B_d(x_0,\rho/2),\ t=t_0+\rho^2/4\},\notag\\
% D_2(x_0,t_0,\rho,\mu)&:=&\{(x,t):\ x\in\Omega^{\mu\rho'}\cap B_d(x_0,\rho),\ t=t_0+\rho^2\},
% \end{eqnarray} are in the closure of the set 
% \begin{eqnarray}\label{sets1}
% D(x_0,t_0,\rho,\epsilon):=\{(x,t):\ x\in \Omega^{\mu\rho'/2}\cap B_d(x_0,(1+\epsilon)\rho),\ t_0+(1-\epsilon)\rho^2/4<t<t_0+\rho^2\}. 
%\end{eqnarray} 
% Furthermore, $D(x_0,t_0,\rho,\epsilon)
%\subset \Omega_T\cap \Gamma_K^+(x_0,t_0,\rho_0,R)$ and the sets $D_1(x_0,t_0,\rho,\mu)$ and $D_2(x_0,t_0,\rho,\mu)$ stay 
% away from the parabolic boundary of  
% $D(x_0,t_0,\rho,\epsilon)$. In particular, 
we can apply the Harnack inequality to conclude that 
\begin{align}  \label{sappo1}
f_1^u(x_0,t_0,\rho,\mu)=&\inf_{D(x_0,t_0,\rho,\mu)} u(x,t)  \\
& \ge 2^{-\gamma_1}\sup_{D(x_0,t_0,\rho/2,\mu)} u(x,t)\geq
2^{-\gamma_1}f_1^u(x_0,t_0,\rho/2,\mu),
\notag\end{align}
for some $\gamma_1=\gamma_1(H,M,r_0,\mu)>0$. In particular, iterating  $k$ times the
inequality in \eqref{sappo1}, where $k$ satisfies $%
2\rho_0>2^{-k}\rho\geq\rho_0$, we see that 
\begin{align*}  \label{sappo2}
f_1^u(x_0,t_0,\rho,\mu)& \geq 2^{-k\gamma_1}f_1^u(x_0,t_0,2^{-k}\rho,\mu)
\\
& \geq 
\biggl (\frac {\rho_0}{\rho}\biggr )^{\gamma_1}
f_1^u(x_0,t_0,2^{-k}\rho,\mu).
\notag\end{align*}
This latter inequality implies the statement in Lemma \ref{lema1}, thus completing the proof. 
\end{proof}

\begin{lemma}\label{lema2} 
Let $K\gg1$ be given, $(x_0,t_0)\in S_T$, and assume that $u$ is a solution to $Hu=0$ in $\Omega_T$
such that $u\geq 0$ in $\Omega_T\cap \Gamma_K^+(x_0,t_0,\rho_0,R)$ for some $%
\rho_0$ and $R$ such that $0<\rho_0\leq R$. Furthermore, assume that 
\begin{equation*}
u(x,t)=0\mbox{ whenever }(x,t)\in \partial_p\Omega_T\setminus
C_{\rho_0/2}(x_0,t_0).
\end{equation*}
Then, there exists $\gamma_2>0$, which depends on $H$, $M$, $K$ and $r_0$,
such that 
\begin{equation*}
f_2^u(x_0,t_0,\rho,K)\leq \biggl (\frac {2\rho_0}\rho\biggr )%
^{\gamma_2}f_2^u(x_0,t_0,\rho_0,K)
\end{equation*}
for all $\rho$ such that $0<\rho_0\leq\rho\leq R$. Moreover, $%
\gamma_2\to\infty$ as $K\to\infty$.
\end{lemma}

\begin{proof}  
By simply using the maximum principle, we first
note that since $u=0$ continuously on $\partial_p\Omega_T\cap\{t: t\leq t_0-\rho_0^2\}$, we also have that $u\equiv
0 $ on $\Omega_T\cap\{t: t\leq t_0-\rho_0^2\}$. Furthermore, again by the maximum
principle, applied to $u$ in $\Omega_T\setminus C_{K\rho,\rho}(x_0,t_0)$, we
see that as a function of $\rho\in[\rho_0,R]$, the $f_2^u(x_0,t_0,\rho,K)$ decreases,  and therefore the conclusion of Lemma \ref{lema2} holds for $\rho\in[\rho_0,2\rho_0]$. Hence it remains to consider $%
\rho\in(2\rho_0,R]$. For such $\rho$ we see that there exists $(\hat x,\hat
t)\in \Omega_T\cap \partial_p C_{K\rho,\rho}(x_0,t_0)\cap \{(x,t):\
|t-t_0|<\rho^2\}$ such that $f_2^u(x_0,t_0,\rho,K)=u^-(\hat x,\hat t)$.
Note, in particular, that $d(x_0,\hat x)=K\rho$ and that $(\hat x,\hat t)\in
C_{2\rho}^+(\hat x,t_0-\rho^2)$. Hence, 
\begin{eqnarray}  \label{sappo3}
f_2^u(x_0,t_0,\rho,K)=u^-(\hat x,\hat t) \leq \sup_{\Omega_T\cap
C_{2\rho}(\hat x,t_0-\rho^2)}u^-.
\end{eqnarray}
We claim that that there exists $\hat K\gg1$, $\hat K\ll K$, $\hat K=\hat
K(H,M)$, such that 
\begin{eqnarray}  \label{sappo3claim}
\sup_{\Omega_T\cap C_{2\rho}(\hat x,t_0-\rho^2)}u^-\leq
\theta\sup_{\Omega_T\cap C_{\hat K\rho,2\rho}(\hat x,t_0-\rho^2)}u^-
\end{eqnarray}
for some $\theta\in (0,1)$. This is proved by arguing as in Lemma \ref{lem4.5-Kyoto1}, except that the proof is simpler: we only need the function $\Phi_1$, and can omit $\Phi_2$, since $u^-$ vanishes on the bottom of $\Omega_T \cap
C_{\hat K\rho,2\rho}(\hat x,t_0-\rho^2)$.

To proceed with the proof of Lemma \ref{lema2} we note that \eqref{sappo3}
and \eqref{sappo3claim} imply that 
\begin{eqnarray*}  \label{sappo3bab}
f_2^u(x_0,t_0,\rho,K)\leq\theta\sup_{\Omega_T\cap C_{\hat K\rho,2\rho}(\hat
x,t_0-\rho^2)}u^-.
\end{eqnarray*}
We next note that: 1) the sets $\Omega_T\cap C_{\hat K\rho,2\rho}(\hat
x,t_0-\rho^2)$ and 
\[
\Omega_T\cap \partial_p C_{K\rho_0,\rho_0}(x_0,t_0) \cap
\{(x,t):\ |t-t_0|<\rho_0^2\}\]
 are separated by the cylindrical surface 
\begin{eqnarray*}  \label{sappo4}
S=\{d(x_0,x)=(K-\hat K)\rho\}=\{d(x_0,x)=qK\rho\},
\end{eqnarray*}
where $q=(K-\hat K)/K\in[1/2,1)$, provided $K\geq 2\hat K$; 2) that 
\[
\Omega_T\cap \partial_p C_{Kq\rho,q\rho}(x_0,t_0) \cap \{(x,t):\
|t-t_0|<(q\rho)^2\}\subset S;
\]
3) and that $u\geq 0$ in 
\[
S\setminus(\Omega_T\cap
\partial_p C_{Kq\rho,q\rho}(x_0,t_0) \cap \{(x,t):\ |t-t_0|<(q\rho)^2\}).
\]
In particular, by the maximum principle, we obtain that 
\begin{eqnarray*}  \label{sappo5}
f_2^u(x_0,t_0,\rho,K)&\leq&\theta\sup_{\{\Omega_T\cap \partial_p
C_{Kq\rho,q\rho}(x_0,t_0) \cap \{(x,t):\ |t-t_0|<(q\rho)^2\}\}}u^-  \notag \\
&=&\theta f_2^u(x_0,t_0,q\rho,K)=q^{\gamma_2}f_2^u(x_0,t_0,q\rho,K)
\end{eqnarray*}
where $\gamma_2=\log_q\theta>0$. Next, we choose $k\geq 1$ so that $\rho_0\leq q^k\rho\leq 2\rho_0$, and by iteration we derive 
\begin{align*}  \label{sappo5b}
f_2^u(x_0,t_0,\rho,K) & \leq q^{k\gamma_2}f_2^u(x_0,t_0,q^k\rho,K)
\\
& \leq \biggl (\frac {2\rho_0}\rho\biggr )^{\gamma_2}f_2^u(x_0,t_0,\rho_0,K).
\notag
\end{align*}
Finally, for $K\geq 2\hat K$ we have 
\begin{eqnarray*}  \label{sappo6}
\frac 1 q&=&1+\frac {\hat K} {K-\hat K}\leq 1+\frac {2\hat K} K,\ \ln
q^{-1}\leq \frac {2\hat K} K,  \notag \\
\gamma_2&=&\log_q\theta\geq \frac{K\ln(\theta^{-1})}{2\hat K}\to\infty%
\mbox{
as }K\to\infty.
\end{eqnarray*}
In particular, this completes the proof of Lemma \ref{lema2}.
\end{proof}

In what follows we let 
\[
\Omega_{[t_0+\rho^2,T]} = \Omega_T\cap\{(y,s)\in \R^{n+1}\mid t_0+\rho^2<s<T\}.
\]
For a given Borel set $E\subset \p_p\Omega_{[t_0+\rho^2,T]}$, we will denote 
by $\omega_{\Omega_{[t_0+\rho^2,T]}}^{(x,t)}(E)$
the value in $(x,t)\in \Omega_{[t_0+\rho^2,T]}$ of the $H$-parabolic measure of $E$.

\begin{lemma}\label{lema3} 
Let $K\gg1$ be given,  $(x_0,t_0)\in S_T$, and suppose that $0<2\rho\leq \nu r_0$, where $\nu>0$ is a
constant. Then, there exist constants $\mu\in (0,1)$ depending on $M$, and $\hat c$ depending on $H$, $M$, $\nu$ and $K$, such that 
\begin{align*}
& \omega^{(x,t)}_{\Omega_{[t_0+\rho^2,T]}}(\Omega_T\cap (\Rn\times\{t_0+\rho^2\}))
\\
& \leq \hat c\ \omega_{\Omega_{[t_0+\rho^2,T]}}^{(x,t)}(\Omega_T^{\mu\rho^{\prime }}\cap (B_d(x_0,\rho)\times
\{t_0+\rho^2\})).
\end{align*}
whenever $(x,t)\in \Omega_T\cap C_{2K\rho}(x_0,t_0)\cap (\Rn\times \{t_0+4\rho^2\})$, where $\rho' = \min(p,r_0)$.
\end{lemma}

\begin{proof} Follows just as the proof of
Lemma 4.5 in \cite{SY}. One also needs to prove the equivalent of Theorem 2.4 in
\cite{SY}, which also follows just as in that article. These proofs make use of
Lemma \ref{lem4.12}, Lemma \ref{lem4.5-ky0} and the Harnack inequality. We omit the details.
\end{proof}

We are finally in a position to establish the main result of this section.

\begin{proof}[Proof of Theorem \ref{T:doubling}]
To start the proof we
choose $K\gg 1$ large enough to guarantee that $\gamma_1<\gamma_2$, where $%
\gamma_1$ and $\gamma_2$ are the constants of Lemma \ref{lema1} and Lemma %
\ref{lema2} respectively. Moreover, for this choice of $\gamma_1$, $\gamma_2 
$, and given the constant in Lemma \ref{lema3}, $\hat c=\hat c(H, M, \nu,K)$%
, $1\leq\hat c<\infty$, we let $\hat r=\hat r(H, M, \nu,K)$ be
\begin{equation}  \label{ban}
\mbox{the smallest 
    $\hat r$ which satisfies $4^{-\gamma_1}(\hat r/4r)^{\gamma_2-\gamma_1}\geq \hat c$}%
.
\end{equation}
Below we will, in the end, distinguish between the cases $\nu r_0\leq \hat r 
$ and $\nu r_0>\hat r$. Let $\mu$ be the constant in Lemma \ref{lema3}. To
prove Theorem \ref{T:doubling} we intend to prove that there exists a constant $%
c=c(H,M,\nu,K)$ such that 
\begin{eqnarray}  \label{the4eq1}
u(x,t):=c\omega^{(x,t)}(\Delta(x_0,t_0,r))-\omega^{(x,t)}(\Delta(x_0,t_0,2r))%
\geq 0,
\end{eqnarray}
whenever $(x,t)\in\Gamma_K^+(x_0,t_0,4r,\nu r_0)$. To start the proof of %
\eqref{the4eq1} we first note, using Lemma \ref{lem4.5-ky0} and the Harnack
inequality, that 
\begin{eqnarray}  \label{the4eq2}
\omega^{(x,t)}(\Delta(x_0,t_0,r))\geq \tilde c^{-1},
\end{eqnarray}
whenever $(x,t)\in \Omega_T^{\mu \rho^{\prime }}\cap (B_d(x_0,2K\rho)\times\{t_0+4\rho^2\})$, $0<2\rho\leq R\leq\nu r_0$, for some $%
\tilde c=\tilde c(H,M,\nu,K,R)$, $1\leq\tilde c<\infty$. Let $\hat c$ be the
constant in Lemma \ref{lema3}. Then, using \eqref{the4eq2} and Lemma \ref%
{lema3} we see that 
\begin{eqnarray}  \label{the4eq3}
\hat c\tilde c\omega^{(x,t)}(\Delta(x_0,t_0,r))&\geq& \hat c
\omega_{\Omega_{[t_0+\rho^2,T]}}^{(x,t)}(\Omega_T^{\mu\rho^{\prime }}\cap (B_d(x_0,\rho)\times\{t_0+\rho^2\}))  \notag \\
&\geq &\omega_{\Omega_{[t_0+\rho^2,T]}}^{(x,t)}(\Omega_T\cap\{t:\
t=t_0+\rho^2\})
\end{eqnarray}
when $(x,t) \in \Omega_T^{\mu \rho ^{\prime }}\cap \{(x,t): x \in
B_d(x_0,2K\rho), t=t_0+4\rho^2\}$. Note that the first inequality in 
\eqref{the4eq3} uses 
\eqref{the4eq2}, the trivial inequality $1\geq
\omega_X(x,t,\Omega_T^{\mu\rho^{\prime }}\cap \{(x,t): x \in B_d(x_0,\rho),
t=t_0+\rho^2\},\Omega_{[t_0+\rho^2,T]})$ and the maximum principle on $%
\Omega_T \cap \{t \geq t_0+\rho^2\}$. Furthermore, let $4r\leq 2\rho\leq
R\leq\nu r_0$. Then, and this is a simple consequence of the maximum
principle, 
\begin{equation}  \label{the4eq4}
\omega_{\Omega_{[t_0+\rho^2,T]}}^{(x,t)}(\Omega_T\cap\{t:\ t=t_0+\rho^2\})\geq
\omega^{(x,t)}(\Delta(x_0,t_0,2r)),
\end{equation}
whenever $(x,t)\in \Gamma_K^+(x_0,t_0,4r,\nu r_0)\cap \{t:\ t=t_0+4\rho^2\}$%
. In particular, combining \eqref{the4eq3} and \eqref{the4eq4} we can
conclude that 
\begin{eqnarray}  \label{the4eq5}
\hat c\tilde c\omega^{(x,t)}\Delta(x_0,t_0,r))\geq
\omega^{(x,t)}\Delta(x_0,t_0,2r)),
\end{eqnarray}
whenever $(x,t)\in \Gamma_K^+(x_0,t_0,4r,\nu r_0)\cap \{t:\ t=t_0+4\rho^2\}$%
. Therefore the function $u$ in \eqref{the4eq1}, defined with constant $%
c=\hat c\tilde c$, satisfies $u\geq 0$ in $\Gamma_K^+(x_0,t_0,4r,\nu r_0)$.
In particular, if $\nu r_0\leq \hat r$, where $\hat r=\hat r(H, M, \nu,K)$
is as in \eqref{ban}, then the constant $\tilde c$, and hence $c$, can be
chosen to only depend on $H, M, \nu,K$, and we are done. Hence, it only
remains to consider the case $\nu r_0> \hat r$. However, by arguing as above, we see in this case that there exists $c=c(H, M, \nu,K)$ such that, if
we consider the function $u$ in \eqref{the4eq1} with this $c$, then 
\begin{itemize}%\label{the4eq6}
\item[(i)] $u(x,t)\geq 1$, for $(x,t)\in \Omega_T^{\mu\rho^{\prime }}\cap
\{(x,t): x \in B_d(x_0,\rho), t=t_0+\rho^2\}$, for $2r\leq\rho\leq 4r$;
\item[(ii)] $u(x,t)\geq 0$ for $(x,t)\in \Gamma_K^+(x_0,t_0,4r,\hat r)$.
\end{itemize}
In the following we prove that %\eqref{the4eq6} 
(i) and (ii) imply \eqref{the4eq1} for
all $(x,t)\in \Gamma_K^+(x_0,t_0,4r,\nu r_0)$. To do this we argue by
contradiction. Hence, we assume that there exist $\rho>4r$ such that $%
u\geq 0$ whenever $(x,t)\in \Gamma_K^+(x_0,t_0,4r,\rho)$ and that $u(\hat
x,\hat t)<0$ at some point 
\[
(\hat x,\hat t)\in \Omega_T\cap \{(x,t): x \in
B_d(x_0,2K\rho), t=t_0+4\rho^2\}\subset \Gamma_K^+(x_0,t_0,4r,2\rho).
\]
Let
 $\omega_{\Omega_T\setminus
C_{K\rho,\rho}(x_0,t_0)}$ denote the $H$-parabolic measure with respect to $%
\Omega_T\setminus C_{K\rho,\rho}(x_0,t_0)$. Then, we first note that 
\begin{eqnarray}  \label{the4eq7++}
u(\hat x,\hat t)&\geq& \int\limits_{\Omega_T^{\mu\rho^{\prime }}\cap
(B_d(x_0,\rho)\times\{t_0+\rho^2\})}u d\omega_{ \Omega_T\setminus C_{K\rho,\rho}(x_0,t_0)}^{(\hat x,\hat
t)}\notag
\\
& +&  \int\limits_{\Omega_T\cap \partial_p C_{K\rho,\rho}(x_0,t_0)\cap
\{(x,t):\ |t-t_0|<\rho^2\}} u d\omega_{ \Omega_T\setminus C_{K\rho,\rho}(x_0,t_0)}^{(\hat x,\hat
t)}.
\end{eqnarray}
Let 
\begin{eqnarray*}\label{the4eq7}
E_1=\omega_{\Omega_T\setminus
C_{K\rho,\rho}(x_0,t_0)}^{(\hat x,\hat t)}(\Omega_T^{\mu\rho^{\prime }}\cap (B_d(x_0,\rho)\times\{t_0+\rho^2\}))
,
\end{eqnarray*}
and 
\begin{eqnarray*}\label{the4eq7}
E_2=\omega_{\Omega_T\setminus
C_{K\rho,\rho}(x_0,t_0)}^{(\hat x,\hat t)}(\Omega_T\cap \partial_p
C_{K\rho,\rho}(x_0,t_0)\cap \{(x,t):\ |t-t_0|<\rho^2\}).
\end{eqnarray*}
Using \eqref{the4eq7++}, Lemma \ref{lema1}
and Lemma \ref{lema2} we deduce that
\begin{eqnarray}\label{the4eq7}
u(\hat x,\hat t)\geq E_1\left(\frac{2r}{\rho}\right)^{\gamma_1}-E_2\left(\frac{8r}{\rho}\right)^{\gamma_2}.
\end{eqnarray}
Furthermore, by the maximum principle and Lemma \ref{lema3} we see that 
\begin{align}\label{the4eq8}
& \omega_{ \Omega_T\setminus C_{K\rho,\rho}(x_0,t_0)}^{(\hat x,\hat t)}(\Omega_T\cap \partial_p C_{K\rho,\rho}(x_0,t_0)\cap
\{(x,t):\ |t-t_0|<\rho^2\}) 
\\
& \leq \omega_{\Omega_{[t_0+\rho^2,T]}}^{(\hat x,\hat t)}(\Omega_T\cap\{t:\
t=t_0+\rho^2\})  
\notag\\
& \leq \hat c \omega_{\Omega_{[t_0+\rho^2,T]}}^{(\hat x,\hat t)}(\Omega_T^{\mu\rho^{\prime }}\cap
\{(x,t): x \in B_d(x_0,\rho), t=t_0+\rho^2\}),
\notag
\end{align}
and so $E_2 \le \hat c E_1$.
In particular, using that $0> u(\hat x,\hat t)$ and combining \eqref{the4eq7} and \eqref{the4eq8} we can conclude that
\begin{equation*}  \label{the4eq9}
4^{-\gamma_2}(\rho/2r)^{\gamma_2-\gamma_1}<\hat c\leq 4^{-\gamma_2}(\hat
r/4r)^{\gamma_2-\gamma_1},
\end{equation*}
and hence that $2\rho< \hat r$. This implies that $ \Gamma_K^+(x_0,t_0,4r,2\rho)
\subset \Gamma_K^+(x_0,t_0,4r,\hat r)$.  Since  $(\hat x,\hat t)\in \Gamma_K^+(x_0,t_0,4r,2\rho)$ we can 
therefore conclude from $(B)$ that $u(\hat x,\hat t)\geq 0$. This contradicts our choice of $(\hat x,\hat t)$ and 
hence \eqref{the4eq1} must be true. This completes the proof of Theorem \ref{T:doubling}. 
\end{proof}

\end{document}